\documentclass{amsart}

\usepackage[utf8]{inputenc}
\usepackage[margin=1in]{geometry}
\usepackage[dvipsnames,svgnames,table]{xcolor}
\usepackage{enumerate, enumitem}
\usepackage{amssymb, amsmath, amsfonts, amsthm, esint, mathtools, nicefrac, bbm}

\usepackage{cite}
\usepackage[colorlinks,linkcolor = blue,citecolor = green]{hyperref}
\usepackage[nameinlink]{cleveref}

\newcommand{\abs}[1]{\left\vert#1\right\vert}
\newcommand{\norm}[1]{\left\|#1\right\|}  
\newcommand{\set}[1]{\left\{ #1 \right\}}
\newcommand{\brak}[1]{\left\langle #1 \right\rangle}

\newcommand{\dd}{\, {\rm d}}

\newcommand{\1}{\mathbbm{1}}

\newcommand{\vv}{\brak{v}}
\newcommand{\vvo}{\brak{v_0}}

\newcommand{\eps}{\varepsilon}

\newcommand{\Id}{\mathrm{Id}}
\newcommand{\tens}{\otimes}
\newcommand{\Tr}{Tr}

\newcommand{\R}{\mathbb R}
\newcommand{\T}{\mathbb T}

\newcommand{\QLFD}{Q_{\rm LFD}}
\newcommand{\QL}{Q_{\rm L}}
\newcommand{\Lrapid}{L^\infty_{{\rm rapid}}}

\newcommand{\cA}{\mathcal A}
\newcommand{\cB}{\mathcal B}
\newcommand{\cC}{\mathcal C}
\newcommand{\cP}{\mathcal P}
\newcommand{\cS}{\mathcal S}

\newcommand{\nfrac}{\nicefrac}
\newcommand{\sfrac}{\nfrac}

\newcounter{num} \numberwithin{num}{section}

\newtheorem{theorem}[num]{Theorem}
\newtheorem{proposition}[num]{Proposition}
\newtheorem{lemma}[num]{Lemma}

\theoremstyle{definition}
\newtheorem{definition}[num]{Definition}

\theoremstyle{remark}
\newtheorem{remark}[num]{Remark}

\numberwithin{equation}{section}

\allowdisplaybreaks

\usepackage{tikz}
\usetikzlibrary{arrows.meta, positioning}
\newcommand{\fmassspreading}{
\begin{figure}
\begin{tikzpicture}[
    gdot/.style={circle, fill=DarkGreen!70!black, inner sep=1pt}
]
    \def\h {3.5}
    \def\b {-1.5}

    \draw[<->] (-2, 0) -- (5, 0) node[right] {\scriptsize $x$};
    \draw[<->] (0, \b-.5) -- (0, {\h+1}) node[above] {\scriptsize $v$};
    
    \fill[thick, blue!70!white] (0,0) circle (0.5);
    \draw[dashed, thick, blue!30!white] (0,0) circle (0.7);
    \fill[opacity=.2, blue] (0,0) circle (0.7);
    \node[blue] at (-.9, -.9) {\scriptsize Step 1};

    \draw[->,thick, red!30!white] (0,.5) -- (0,\h-.4);
    \draw[->,thick, red!30!white] (0.5,0) -- (.3,\h-.15);
    \draw[->,thick, red!30!white] (-.5,0) -- (-.3,\h-.15);
    \fill[-,thick,  red!70!white] (0,\h) circle (.3);
    \node[red,rotate=86] at (-.6,\h/2+.1) {\scriptsize Step 2};

    \draw[->,thick, violet!30!white] (.4,\h) -- (3.55,\h);
    \draw[->,thick, violet!30!white] (.1,\h+.3) -- (3.7,\h+.15);
    \draw[->,thick, violet!30!white] (.1,\h-.3) -- (3.7,\h-.15);
    \fill[-, thick, violet!70!white] (3.8,\h) circle (.15);
    \node[violet,rotate=-3] at (2,\h+.4) {\scriptsize Step 3};
    \node[violet] at (4.5,\h) {\scriptsize $(x_e, \frac{2x_e}{t_e})$};
    
    \draw[->,thick, DarkGreen!30!white] (3.8, \h-.25) -- (3.8, \b+.1) node[midway, DarkGreen,above, sloped] {\scriptsize Step 4};
    \node[gdot] at (3.8, \b) {};
    \node[DarkGreen] at (4.5, \b) {\scriptsize $(x_e, v_e)$};
\end{tikzpicture}
\caption{A cartoon of the four main steps of \Cref{p.nonvacuum}, taking positivity from the origin to $(x_e,v_e)$ over time $t_e$. Each colored ball represents a set on which we establish positivity of $f(1-f)$ in the corresponding step. Step 1: we \emph{sustain} an ellipticity core that allows one to spread in $v$ over a time interval $[0,t_{\rm core}]$; this uses the transport lemma. Step 2: we \emph{spread} in $v$ by diffusion to an intermediate velocity $\sfrac{2 x_e}{t_e}$; this uses the diffusion lemma. Step 3: we \emph{sustain} mass moving in $x$ by transport at velocity $\sfrac{2x_e}{t_e}$ over time $\sfrac{t_e}{2}$ to make it to $x=x_e$; this uses the transport lemma.  Step 4: we \emph{spread} in $v$ by diffusion from $\sfrac{2x_e}{t_e}$ to $v_e$; this uses the diffusion lemma.}
\label{f.mass_spreading}
\end{figure}
}

\author{William Golding and Christopher Henderson}
\title[]{Global smooth solutions to the inhomogeneous Landau--Fermi--Dirac equation}
\date{\today}							
\address[William Golding]{\newline Department of Mathematics, \newline The University of Chicago, Chicago, IL 60637, USA}
\email{wgolding@uchicago.edu}
\address[Christopher Henderson]{\newline Department of Mathematics, \newline The University of Maryland, College Park, MD 20742, USA}
\email{ckhend@umd.edu}
\thanks{\textbf{Funding:} C.H. is supported by NSF grants DMS-2617615 and DMS-2204615.}
\thanks{\textbf{AI Acknowledgments:} ChatGPT was used in the final stages of preparation of the manuscript for copyediting.}

\begin{document}

\begin{abstract}
We consider the spatially inhomogeneous Landau--Fermi--Dirac equation with Coulomb potential, a quantum modification of the classical Landau equation for fermions.  Mathematically, the Pauli exclusion principle manifests as an additional \emph{a priori} $L^\infty$-bound for solutions. Using this bound, we propagate polynomial decay in velocity, yielding unconditional upper bounds on the local mass and energy densities, thereby ruling out the possibility of implosions in the hydrodynamic quantities. Combining this estimate with a modified mass spreading method that yields desaturation, we deduce the existence of global-in-time classical solutions for rough initial data with polynomial decay in velocity. This result stands in stark contrast to the theory for the classical Landau and Boltzmann equations, for which no comparable nonperturbative global existence result is known despite sustained effort. 

Our treatment is almost entirely self-contained, using only robust, generic estimates for linear kinetic equations. In particular, our proof of local existence, in contrast to prior works, more closely mirrors the theory for parabolic equations using weak solutions and simpler function spaces. It may provide a concise roadmap to organizing, adapting, and applying the various linear and nonlinear estimates to obtain well-posedness for kinetic equations.
\end{abstract}

\maketitle

\section{Introduction}

We study the \emph{inhomogeneous} Landau--Fermi--Dirac (LFD) equation, a variant of the classical Landau equation in which the collision operator is modified to account for fermionic quantum statistics. The distinguishing feature of this model is the Pauli exclusion principle: no two identical fermions can occupy the same quantum state. At the kinetic level, this principle manifests as an upper bound of $\eps^{-1}$ on the allowed number of fermions per unit of phase space. Mathematically, this translates into an additional {\em a priori} pointwise bound on solutions to LFD: if $0 \le f_{\rm in} \le \eps^{-1}$, then 
\begin{equation}\label{eq:Pauli_exclusion}\tag{Pauli Exclusion}
    0 \le f(t,x,v) \le \eps^{-1}\qquad\text{for each }\;(t,x,v) \in [0,T]\times \Omega \times \R^3.
\end{equation}
Here $\eps > 0$ is a parameter related to $\hbar$ that measures the strength of quantum effects in the system. 

Our main result shows that this added structure can be leveraged to obtain global-in-time existence of smooth solutions. More precisely, in \Cref{thm:upper_bounds} below, we show that \eqref{eq:Pauli_exclusion} is sufficient to propagate decay and derive global-in-time \emph{a priori} bounds on the local mass and energy densities
\begin{equation*}
    \rho(t,x) \coloneqq \int_{\R^3} f(t,x,v) \dd v \qquad \text{and} \qquad E(t,x) \coloneqq \int_{\R^3} \abs{v}^2 f(t,x,v) \dd v.
\end{equation*}
For the classical Landau equation, such estimates, together with the existing continuation criteria, would be sufficient for global regularity; see, for example, \cite{HST2020landau}. The LFD equation, however, has an additional degeneracy at saturated states, where $f = \eps^{-1}$. We overcome this difficulty through a modified mass spreading argument that yields quantitative \emph{desaturation}; see \Cref{t.nonvacuum} below. Together, these two \emph{a priori} estimates provide the basis for our global existence result, \Cref{thm:global_existence} below.

We emphasize that there is a substantial gap between the $L^\infty$ bound provided by Pauli exclusion and global regularity. Although Pauli exclusion prevents pointwise blowup of $f$, it does not prevent singularity formation through loss of decay of $f$ at large velocities: $f$ may remain uniformly bounded while its velocity tails grow sufficiently large to cause singularity formation in $\rho$ and $E$. In other words, $L^\infty$ only controls $L^1$ locally in $v$, and \eqref{eq:Pauli_exclusion} gives no direct control over the mass and energy densities.

This is a genuine obstruction, not merely a technical one. In closely related kinetic models, such as the four-dimensional gravitational Vlasov--Poisson equation \cite{Horst_Vlasov,LemouMehatsRaphael} and the classical Landau equation with $\gamma \in (\sqrt{3},2]$ \cite{Monster_paper}, implosions driven by this mechanism---often referred to as \emph{tail fattening}---are known to occur even while $f$ remains bounded up to the implosion time. The analogous question of whether pointwise control can prevent tail fattening in the classical Landau and Boltzmann equations is studied in related work by the authors and Silvestre \cite{GoldingHendersonSilvestre}. The main novelty of the present work is that, for the LFD equation, \eqref{eq:Pauli_exclusion} can nonetheless be upgraded to bounds on the mass and energy densities, thereby ruling out implosion-type singularity formation entirely.

\subsection{The Landau--Fermi--Dirac Equation}

We now introduce the precise kinetic equation and provide a brief literature review for this less studied model. We consider the Cauchy problem:
\begin{equation}\label{e.LFD}\tag{LFD}
\begin{aligned}
    (\partial_t + v\cdot \nabla_x) f &= \QLFD(f)
    \qquad &&\text{in }(0,T)\times \Omega \times \R^3,\\
    f(0,x,v) &= f_{\rm in}(x,v) \qquad &&\text{on }\Omega\times \R^3,
\end{aligned}
\end{equation}
where $f(t,x,v)$ is an unknown density function, $t \in \R^+$ is time, $x\in \Omega$ is position, and $v \in \R^3$ is velocity. 
Throughout the paper, the spatial domain $\Omega$ is taken to be either $\Omega = \T^3$ or $\Omega = \R^3$ to avoid the delicate issues that arise in the presence of boundaries. 
Whether the results presented below persist when $\Omega$ is a bounded domain is an interesting open question. The operator $(\partial_t + v\cdot\nabla_x)$ models free transport, whereas $\QLFD$ is the nonlinear, nonlocal collision operator
\begin{equation}\label{e.QLFD_eps}
    \QLFD(f)
    \coloneqq \frac{1}{8\pi} \nabla_v \cdot \int_{\R^3} \frac{\Pi(v-v_*)}{\abs{v-v_*}}\left[(1-\eps f_*)f_*\nabla_v f - (1-\eps f)f\nabla_{v_*} f_* \right] \dd v_*,
\end{equation}
where $f_* = f(t,x,v_*)$ and $\Pi(z)$ is the usual projection onto $z^\perp$: 
\begin{equation*}
    \Pi(z) = \Id - \frac{z\tens z}{\abs{z}^2}.
\end{equation*}
The fermionic collision operator $\QLFD$ can be traced back to Nordheim~\cite{nordheim1928kinetic} and Uehling and Uhlenbeck~\cite{uehling1933transport}; see also Chapman and Cowling's book~\cite[Chapter~17]{chapmancowling}. An analogous modification exists for the Boltzmann collision operator; see, for example, \cite{Dolbeault}. We do not consider that model here. 

Formally, one recovers the classical Landau collision operator $\QL$ with Coulomb potential by setting $\eps = 0$; see~\cite{Sampaio2} for a rigorous treatment of this limit. For fixed $\eps > 0$, the parameter $\eps$ may be removed by rescaling, and we henceforth normalize $\eps = 1$. Under this normalization, the Pauli bound becomes $0 \le f \le 1$, and it is convenient to introduce 
\begin{equation}\label{eq:tilde_f}
    \tilde f \coloneqq f(1-f).
\end{equation}
In terms of $\tilde f$, the collision operator can be written in collisional, divergence, and nondivergence forms:
\begin{equation}\label{e.QLFD}
\begin{split}
    \QLFD(f)
    &= \frac{1}{8\pi} \nabla_v \cdot \int_{\R^3} \frac{\Pi(v-v_*)}{\abs{v-v_*}}\left[\tilde f_*\nabla_v f - \tilde f\nabla_{v_*} f_* \right] \dd v_*\\[5pt]
    &= \nabla_v\cdot \Big( A[\tilde f] \nabla_v f - (\nabla a[f]) \tilde f\Big),\\[5pt]
    &= \Tr(A[\tilde f] D^2_v f)+ b[f] \cdot \nabla_v f + f\tilde f,
\end{split}
\end{equation}
with nonlocal coefficients defined by the following formulas:
\begin{equation}\label{e.QLFD_coefficients}
    A[g] \coloneqq \frac{\Pi(\cdot)}{8\pi\abs{\cdot}} \ast g, \qquad a[g] \coloneqq \frac{1}{4\pi\abs{\cdot}} \ast g, \qquad \text{and} \qquad b[g] \coloneqq \nabla a[\tilde g] - (1-2g)\nabla a[g].
\end{equation}
The representations of $\QLFD$ as an elliptic operator reveal a key structural difference from the classical Landau equation. The classical diffusion matrix is $A[f]$, whereas the LFD diffusion matrix is $A[\tilde f]$. Consequently, the ellipticity of $\QLFD$ is tied to quantitative lower bounds on $\tilde f = f(1-f)$, rather than on $f$ alone. Indeed, if, for fixed $(t,x)$, one has $f(t,x,\cdot)=\1_S$ for some set $S\subset\R^3$, then $\tilde f(t,x,\cdot)=0$ and hence $A[\tilde f]=0$ as well. Thus, one must establish quantitative separation not only from vacuum but also from the saturated state $f=1$.

\subsubsection{Previous Work on \eqref{e.LFD}}

Like the classical Landau equation, \eqref{e.LFD} formally conserves mass, momentum, and energy, and satisfies an $H$-theorem. However, the fermionic modification affects the functional form of the steady states and the entropy structure of the equation; see \cite{BaglandLemou_equilibrium} for a rigorous treatment. Although these structural features do not enter directly into our arguments, they have motivated a substantial part of the existing Landau-Fermi-Dirac literature.

Most of the existing literature on~\eqref{e.LFD} has focused on the spatially homogeneous case; that is, when $f$ is $x$-independent. Many aspects of the classical theory have been extended to the LFD equation with hard or moderately soft potentials, including spectral gap estimates, well-posedness, smoothing, and convergence to equilibrium~\cite{Lemou_spectral_gap,Bagland1,AlonsoBaglandDesvillettesLods,AlonsoBaglandDesvillettesLods_entropy_dissipation,AlonsoBaglandLods_hard_potentials}.
In these regimes, regularity estimates have been obtained uniformly in the quantum parameter $\eps$ ~\cite{AlonsoBaglandLods_hard_potentials,AlonsoBaglandDesvillettesLods}. For the Coulomb case, global existence for the homogeneous LFD equation was resolved in~\cite{GoldingGualdaniZamponi} prior to the corresponding classical problem in ~\cite{GuillenSilvestre}. The mechanisms for well-posedness, however, are vastly different in the classical and quantum Coulomb cases. In particular, it remains largely open whether regularity estimates can be made uniform in $\eps$ in the Coulomb case, and whether LFD admits a monotone Fisher information functional analogous to that of the classical equation.

In the inhomogeneous setting, the only large-data Cauchy theory presently available, to our knowledge, is Sampaio's construction of global distributional solutions \cite{Sampaio1} within the DiPerna--Lions framework \cite{diperna1989cauchy,lions1994boltzmannlandau,villani1996global}. The key point in \cite{Sampaio1} is that the pointwise bound $0\le f(t,x,v) \le 1$ can be leveraged to construct distributional solutions, rather than merely renormalized solutions. This theory does not, however, provide smoothness, uniqueness, or convergence to equilibrium.

\subsection{Main results}

\subsubsection{Propagation of Decay}

Our first main result shows that the Pauli exclusion bound propagates polynomial decay in velocity. For $m\in\R$, we define the $L^\infty_m$ norm via
\begin{equation*}
    \norm{g}_{L^\infty_m} \coloneqq \norm{\brak v^m g}_{L^\infty_{x,v}} \qquad \text{where} \qquad \brak v \coloneqq \left(1+\abs v^2\right)^{1/2}.
\end{equation*}
The precise class of classical solutions used in the statement is defined in \Cref{sec:upper_bounds}. The additional qualitative decay imposed there is used only to justify the comparison argument and does not enter the estimate below.

\begin{theorem}\label{thm:upper_bounds}
    Fix any $f_{\rm in} \in L^\infty_m$ for some $m > 2$ satisfying $0 \le f_{\rm in} \le 1$. Any classical solution $f\colon[0,T]\times \Omega \times \R^3 \to [0,1]$ to \eqref{e.LFD} with initial data $f_{\rm in}$ satisfies the pointwise bound
    \begin{equation*}
        \abs{f(t,x,v)} \le e^{Kt} \frac{\norm{f_{\rm in}}_{L^\infty_m}}{\abs{v}^m} \qquad \text{for each }(t,x,v) \in [0,T]\times \Omega\times \R^3.
    \end{equation*}
    Here $K = K(m)$ depends only only on $m$. The spatial domain $\Omega$ may be either $\T^3$ or $\R^3$.
    Consequently, the mass and energy densities satisfy
    \begin{equation}\label{eq:mass_energy}
        \rho(t,x) = \int_{\R^3} f(t,x,v) \dd v \lesssim  e^{Kt} \|f_{\rm in}\|_{L^\infty_m}
        \qquad \text{and} \qquad E(t,x) = \int_{\R^3} \abs{v}^2 f(t,x,v) \dd v \lesssim e^{Kt}\|f_{\rm in}\|_{L^\infty_m},
    \end{equation}
    provided $m > 5$.
\end{theorem}

\Cref{thm:upper_bounds} uses Pauli exclusion to propagate polynomial decay in velocity from the initial data for the entire lifespan of a classical solution. Consequently, it excludes implosion singularities, by which we mean finite-time singularity formation through concentration of the local mass or energy density.

The significance of \Cref{thm:upper_bounds} is particularly clear in comparison with the classical inhomogeneous equations. No comparable \emph{unconditional} control of the local mass and energy densities is known for the classical Landau or Boltzmann equations. Such a result would imply global regularity by existing continuation criteria; see, for example, \cite{HST2020landau}. 
Moreover, \Cref{thm:upper_bounds} is not invariant under the scaling used to normalize $\eps = 1$. Indeed, after reintroducing $\eps > 0$, the constants in \eqref{eq:mass_energy} degenerate in the semiclassical limit $\eps\to0^+$, and the estimate does not pass to the classical Landau equation.

The proof of \Cref{thm:upper_bounds} is a comparison argument with polynomial barriers, closely related to the method used in \cite{GoldingHendersonSilvestre}. The crux of the argument is a sufficiently sharp estimate of the effect of collisions at a contact point with the barrier. Standard convolution estimates for the coefficients in \eqref{e.QLFD_coefficients} are not precise enough, since their decay is limited by the slowly decaying kernels. The key insight is that, at a contact point between $f$ and the barrier, the nonlocal contributions limiting the standard estimates occur with a favorable sign and may be discarded. Identifying this sign is the essential ingredient in \Cref{thm:upper_bounds}.

The proof also explains the special role of the Coulomb potential. Replacing the factor $\abs{v-v_*}^{-1}$ in \eqref{e.QLFD} by a general power-law kernel $ \abs{v-v_*}^{2+\gamma}$ as in \cite{AlonsoBaglandDesvillettesLods} formally leads to the estimate
\begin{equation}\label{eq:other_gamma}
\norm{f(t)}_{L^\infty_m}\lesssim\exp\left(K\int_0^t\norm{f(s)}_{L^\infty_{3+\gamma}}\dd s\right)\norm{f_{\rm in}}_{L^\infty_m}.
\end{equation}
For the Coulomb potential, $\gamma=-3$, the quantity in the exponential is precisely the unweighted $L^\infty$-norm provided by Pauli exclusion. For $\gamma>-3$, the same argument requires a positive velocity weight that is not supplied by the Pauli bound. 

The failure of \eqref{eq:other_gamma} to close for larger $\gamma$ does not by itself imply singularity formation. It is, however, consistent with the recent construction in \cite{Monster_paper} of finite-time singularities for the classical Landau equation with very hard potentials $\gamma\in(\sqrt3,2]$. In that construction, the distribution remains uniformly bounded up to the singular time while its hydrodynamic fields undergo an implosion. We conjecture that analogous singularities may occur for the LFD equation with sufficiently hard potentials. Determining the range of $\gamma$ for which Pauli exclusion is sufficient to yield global regularity remains an interesting open problem.

Finally, the exponential growth in \eqref{eq:mass_energy} is sufficient to rule out finite-time blowup, but does not provide uniform-in-time control of the mass and energy densities. Consequently, \eqref{eq:mass_energy} does not preclude infinite-time blowup, and the global solutions we construct below (see \Cref{thm:global_existence}) do not fall within the standard framework for convergence to equilibrium developed by Desvillettes and Villani~\cite{DesvillettesVillani_convergence}. Improving the time dependence of \eqref{eq:mass_energy} appears to be a necessary first step for an unconditional theory of the long-time behavior of solutions to \eqref{e.LFD}.

\subsubsection{Mass Spreading and Desaturation}

Our second main result is an {\em a priori} estimate that addresses the remaining degeneracy of the diffusion matrix $A[f(1-f)]$. The Pauli bound $0\leq f\leq1$ alone does not provide ellipticity, since $f(1-f)$ vanishes at both vacuum states and saturated states. To state the result, we isolate the local nondegeneracy required of the initial data.
\begin{definition}\label{d.nonvacuum}
    We say that a measurable function $g\in L^\infty(\Omega \times \R^3)$ satisfying $0 \le g(x,v) \le 1$ is \emph{non-vacuum} at $(x_0,v_0)$ of size $0 < \delta \le \sfrac12$ and radius $r>0$ if 
    \begin{equation}\label{e.nonvacuum}
    \delta \leq g(x,v)
        \leq 1 - \delta \qquad \text{for each }(x,v) \in B_r(x_0)\times B_r(v_0).
\end{equation}
\end{definition}
Although we call this a non-vacuum condition, it also imposes quantitative separation from saturation. The theorem shows that this local condition spreads throughout phase space for positive times.

\begin{theorem}\label{t.nonvacuum}
    Fix any $m > 2$ and $f_{\rm in} \in L^\infty_m$ satisfying $0 \le f_{\rm in} \le 1$ and further assume that $f_{\rm in}$ is non-vacuum at some $(x_0,v_0)$ with radius $r > 0$ and size $\delta > 0$. 
    Let $f:[0,T]\times \Omega \times \R^3 \to \R^+$ be any classical solution of~\eqref{e.LFD} with initial data $f_{\rm in}$. Then, for every $t_0 \in (0,T]$ and $R>0$, there is $c_0 \in (0,\sfrac12)$ such that
    \begin{equation}
    \label{e.nonvacuum_on_compacts}
        0 < c_0 \le  f
            \le 1-c_0
            <  1
            \qquad\text{ on }
            [t_0,T]\times B_R(x_0)\times B_R(v_0).
    \end{equation}
    The constant $c_0$ depends on the parameters $v_0$, $t_0$, $T$, $R$, $r$, $\delta$, $m$, and $\|f_{\rm in}\|_{L^\infty_m}$. Consequently,
    \begin{equation}\label{e.coercivity}
       A[f(1-f)](t,x,v) \ge \lambda_0\left(\frac{\Pi(v)}{\brak{v}} + \frac{\Id - \Pi(v)}{\brak{v}^3}\right) \qquad \text{where} \quad (t,x,v) \in [t_0,T]\times B_R(x_0)\times \R^3.
    \end{equation}
    Here $\lambda_0 > 0$ depends on the same parameters as $c_0$. 
    Further, if $R<r$, \eqref{e.nonvacuum_on_compacts}-\eqref{e.coercivity} hold even up to $t_0 = 0$, with constants that also depend on $r-R$ and deteriorate as $R \uparrow r$.
\end{theorem}

To interpret the result, set $h=1-f$. Physically, $h$ measures the number of available states, or holes. \Cref{t.nonvacuum} shows that if $f$ and $h$ are quantitatively positive on the same phase-space ball initially, then lower bounds for both spread throughout phase space for all positive times, uniformly on compact sets. The bound for $f$ provides mass spreading, while the bound for $h$ provides desaturation; together, they yield coercivity of the diffusion matrix $A[f(1-f)]$.

The proof adapts the mass spreading method developed for the classical Landau equation in \cite{HST2018landau} and refined in \cite{HST2020boltzmann,HSTT_vacuum}. Earlier lower bounds for the Boltzmann equation under stronger assumptions appear in \cite{Briant_nonvacuum,Mouhot_vacuum}.
For a uniformly parabolic equation, instantaneous positivity can be understood as lower bounds on a fundamental solution. The analogous principle holds for linear kinetic equations, because transport and diffusion interact hypoelliptically to create positivity. The nonlinear difficulty is that the diffusion can initially be degenerate: the lower bound must create the ellipticity needed for its own propagation. The argument therefore alternates transport and velocity spreading, using an existing positive region to generate diffusivity and then enlarge that region. Nonlocality is advantageous here, since positivity in one velocity region influences the diffusion matrix at every velocity, thereby preventing the type of persistent free boundary that occurs in problems with local diffusion.

For~\eqref{e.LFD}, the particle and hole bounds must be propagated simultaneously. More precisely, the densities $f$ and $h$ solve a pair of strongly coupled linear kinetic equations with the same diffusion matrix $A[fh]$. Mechanically, existing linear mass spreading bounds are applied iteratively to both $f$ and $h$ individually. The novelty is that the iterative argument still closes when applied to the coupled particle--hole system.

Together, \Cref{thm:upper_bounds} and \Cref{t.nonvacuum} provide upper and lower bounds on the ellipticity of the collision operator $\QLFD$ necessary to apply regularity theory to \eqref{e.LFD}. Crucially, the bounds for positive times depend only on the initial data. Finally, although stated for classical solutions, the estimates are robust and hold uniformly under suitable viscous approximations used in the existence argument.

\subsubsection{Global-in-Time Smooth Solutions}

We now explain how the preceding \emph{a priori} estimates lead to global existence of solutions. To keep the construction simple and focused on the new mechanism, we specialize to $\Omega = \T^3$ with rapidly decaying initial data.

\begin{theorem}\label{thm:global_existence}
    Suppose that $f_{\rm in} \in L^\infty_m(\T^3\times \R^3)$ for every $m > 2$, that $0 \le f_{\rm in} \le 1$, and that $f_{\rm in}$ is non-vacuum at some $(x_0,v_0)$ with radius $r > 0$ and size $\delta > 0$.  Then there exists a global-in-time solution $f\colon [0,\infty) \times \T^3\times \R^3 \to [0,1]$ of \eqref{e.LFD} such that
    \begin{equation}
        f(t) \in \cS(\T^3\times \R^3) \quad \text{for each }t > 0 \qquad \text{and} \qquad \lim_{t\to 0^+} \norm{f(t) - f_{\rm in}}_{L^1_{x,v}} = 0.
    \end{equation}
    Moreover, the solution satisfies the a priori estimates from \Cref{thm:upper_bounds} and \Cref{t.nonvacuum}.
\end{theorem}

The assumption $0 \le f_{\rm in}\le 1$ is the natural Pauli exclusion bound, not a smallness condition; in particular, the initial data can be arbitrarily large in the weighted norms $L^\infty_m$. Thus, for a class of large initial data, \Cref{thm:global_existence} yields a global solution to LFD that is smooth for all positive times.

Results of this type are exceedingly rare for \emph{spatially inhomogeneous} kinetic equations of Landau or Boltzmann type. On the one hand, results for global smooth solutions are generally confined to perturbative regimes---near equilibrium, vacuum, or a spatially homogeneous state---where the dynamics can be controlled through linearization; see, for example, \cite{guo2002periodic,gressman2011boltzmann,luk2019vacuum}. On the other hand, results for large-data are generally confined to renormalized or weak solutions; see, for example, \cite{villani1996global,diperna1989cauchy,Dolbeault,Sampaio1}. 
Moving beyond these regimes requires strong, genuinely nonlinear \emph{a priori} estimates that are typically unavailable.  
To our knowledge, the only other comparable large-data result for an inhomogeneous Landau-type model is for the fuzzy Landau equation studied in \cite{gualdani2025fuzzy}, in which collisions are delocalized in $x$. 

For LFD, the necessary nonlinear \emph{a priori} estimates are provided by \Cref{thm:upper_bounds} and \Cref{t.nonvacuum}, both of which rely essentially on the nonlinear and nonlocal structure of $\QLFD$. The remaining difficulty in \Cref{thm:global_existence} is designing and organizing an approximation scheme compatible with the nonlinear structure. Our construction proceeds in four steps:
\begin{enumerate}[label=\emph{Step \arabic*:}, leftmargin=*, itemsep=.35\baselineskip, parsep=0pt, topsep=.4\baselineskip]
    \item \emph{Local existence.} 
    We add an artificial viscosity term $\kappa\Delta_vf$ to \eqref{e.LFD} and use the Schauder fixed-point theorem to construct a local solution for each $\kappa > 0$. At this stage, we rely on generic estimates of the nonlocal coefficients and a corresponding maximum principle bound; the resulting solutions admit a continuation criterion in terms of a weighted $L^\infty$-norm. 

    \item \emph{Continuation.}
    For each $\kappa > 0$, qualitative regularity and decay of the viscous solution justifies application of the nonlinear \emph{a priori} estimates. Applying \Cref{thm:upper_bounds} controls the continuation norm uniformly in $\kappa$ and extends each viscous solution globally in time.

    \item \emph{Vanishing viscosity limit.}
    The coercivity from \Cref{t.nonvacuum}, combined with the standard $L^2$-energy estimate, yields $L^2_{\rm loc}$-control of $\nabla_vf$ uniformly as $\kappa \to 0^+$. A compactness argument then produces a global weak solution of~\eqref{e.LFD} that is compatible with regularity theory\footnote{By this we mean that truncations of $f$ at large velocities are sufficiently regular to be used as test functions in the weak formulation}.

    \item \emph{Hypoelliptic regularization.}
    Existing regularization estimates for linear kinetic equations and a bootstrap yield regularity of the constructed weak solution. See \Cref{prop:higher_regularity} for a more precise statement for solutions with slower decay in velocity.
\end{enumerate}
Although each step is standard once the corresponding \emph{a priori} estimates are available, the above proof structure may be useful for related models. The outlined construction clearly separates local existence and continuation from quantitative regularization: continuation only uses decay and coercivity estimates, while Schauder estimates and regularity theory enter only after a weak solution has been obtained. In this way, the argument entirely avoids high-order weighted Sobolev spaces used in many earlier Landau and Boltzmann constructions; see the seminal works of the AMUXY group~\cite{morimoto2015polynomial,amuxy2011bounded} and later works~\cite{henderson2021existence,HST2018landau,HendersonWang, HST2022irregular,HST2020landau}.

While we establish global existence and regularity, we make no uniqueness claim for the rough data covered by \Cref{thm:global_existence}. Inhomogeneous kinetic equations present substantial difficulties in proving uniqueness for irregular data; for a detailed discussion, see \cite{henderson2021existence}.  The difficulties in establishing uniqueness for~\eqref{e.LFD} appear to be the same as those of the classical Landau equation. For the classical equation, uniqueness is known once the initial data is regular enough, such as H\"older continuity of any order, and with suitable decay~\cite{HST2020landau}. There are no obvious obstructions to adapting the existing, somewhat onerous, theory to~\eqref{e.LFD}.  Hence, we do not pursue uniqueness further here.

Finally, the restriction to $\Omega = \T^3$ concerns the existence construction rather than the preceding \emph{a priori} estimates, which also hold for $\R^3$. Extending the construction to $\R^3$ requires straightforward technical modifications such as the addition of spatial weights or the use of ``uniformly local'' spaces {\em \`a la} AMUXY~\cite{amuxy2011bounded}. The assumption of arbitrarily high velocity decay similarly avoids tracking the loss of velocity moments during the regularity bootstrap.

\subsection{Notation}\label{ss.notation}

Throughout, we use several standard notations. First, for any vectors $v \in \R^3$, we define
\begin{equation}
    \vv \coloneqq \sqrt{ 1 + |v|^2} \qquad \text{and} \qquad \hat{v} \coloneqq \frac{v}{\abs{v}}.
\end{equation}
We emphasize that the Japanese bracket is almost always applied to the velocity variables. For example, $\brak{\cdot}^mg$, this denotes the function $(t,x,v) \mapsto \vv^m g(t,x,v)$.
Second, we define a standard family of weighted Lebesgue spaces, with a weight only in velocity. Fix a set $\widetilde\Omega \subset \R^3_x \times \R^3_v$ or $\R_t \times \R^3_x \times \R^3_v$. For $p \in [1,\infty]$ and $k\in \R$, we define
\begin{equation}
    L^p_k(\Omega) \coloneqq \left\{g: \|g\|_{L^p_k(\Omega)} \coloneqq \|\brak{\cdot}^k g\|_{L^p(\Omega)} < \infty\right\}.
\end{equation}
We nearly always suppress the dependence on $\Omega$, writing simply $L^p_k$. 
Additionally, we introduce the space of rapidly decaying functions:
\begin{equation}\label{e.Lrapid}
    \Lrapid \coloneqq  \bigcap_{m > 0} L^\infty_m.
\end{equation}
Iterated Lebesgue spaces are written in the standard manner. A constant $C$ is used to denote a constant that may change from line to line, but for which any larger constant may substitute. A constant $c$ is used for smaller constants. Important or fixed constants will have an adornment. Whenever feasible, we simply use the convention $A \lesssim B$ to denote $A \leq C B$.

\subsubsection{Kinetic H\"older Spaces}

Notationally, we denote a point $z = (t,x,v)$.  If $z$ is decorated, then its coordinates will be decorated in the same way.  For example, $\tilde z = (\tilde t, \tilde x, \tilde v)$. 
To state local results, we introduce kinetic distance, cylinders, and H\"older spaces.  There are several choices in the literature, although these yield equivalent norms and functional spaces.  We opt for the simplest-to-state spaces here because the particular form plays almost no role in this paper.

First, the kinetic distance is defined via
    \begin{equation} \label{defn:kinetic_distance}
        d(z,z') = \min_w\left\{\max\left\{ \abs{t' - t}^{1/2}, \abs{x' - x - (t'-t)w}^{1/3}, \abs{v'-w}, \abs{v - w} \right\}\right\}.
    \end{equation}
    A slightly simpler choice is often made by selecting $w=v$ in the formula above; however, this is not a metric.  We point the interested reader to the in-depth discussion in~\cite[Section 2A]{imbert2018schauder} regarding $d$ and its relation to kinetic equations. From here, we define the kinetic cylinder with ``radius'' $R$ and base point $z_0$ by
    \begin{equation}
    \label{defn:kinetic_cylinders}
        Q_R(z_0)
        \coloneqq \set{z : d(z,z_0) \leq R, t \leq t_0}.
    \end{equation}
    If the base point $z_0$ is omitted, it is understood to be the origin: $Q_R = Q_R(0)$. Note that $Q_R(z_0)$ is not a Cartesian product. Given a set $\mathcal Z \subset \R^7$ and $\alpha \in (0,1)$, we define the basic kinetic H\"older seminorm by:
    \begin{equation}
        [g]_{C_{\rm kin}^{0,\alpha}(\mathcal Z)} \coloneqq \sup_{\substack{z \neq z',\\z,\,z'\in \mathcal Z}} \frac{\abs{g(z) - g(z')}}{d(z,z')^\alpha}.
    \end{equation}
    We note an alternative (equivalent) characterization in terms of kinetic polynomials; see~\cite{imbert2018schauder,imbert2021nonlinear}.
    The most important higher order seminorm in our work is the second order one:
    \begin{equation}
        [g]_{C_{\rm kin}^{2,\alpha}(\mathcal Z)} = [(\partial_t +v\cdot\nabla_x) f]_{C^{0,\alpha}_{\rm kin}(\mathcal Z)}
            + [D^2_vf]_{C^{0,\alpha}_{\rm kin}(\mathcal Z)}.
    \end{equation}
    The higher order spaces $C^{k,\alpha}_{\rm kin}(\mathcal Z)$, for $k>2$, may be characterized similarly.  We omit the details and refer to previous works instead.

\subsection{Organization of the paper} 

We begin in \Cref{sec:preliminaries} by collecting useful bounds on the nonlocal coefficients in~\eqref{e.LFD}  The next two sections establish the main nonlinear \emph{a priori} estimates: In \Cref{sec:upper_bounds}, we prove the global upper bound in \Cref{thm:upper_bounds}, which contains the central new argument of the paper. In \Cref{sec:nonvacuum}, we prove the mass-spreading and desaturation result in \Cref{t.nonvacuum}. We then turn to the construction of solutions: In \Cref{sec:existence}, we introduce the viscous approximation, construct local-in-time approximate solutions, use \Cref{thm:upper_bounds} to continue them globally, and pass to the vanishing viscosity limit using \Cref{t.nonvacuum}. In \Cref{sec:regularity}, we apply linear hypoelliptic regularity estimates and a bootstrap to obtain smoothness and complete the proof of \Cref{thm:global_existence}. Finally, \Cref{appendix:linear} records several auxiliary results for linear Kolmogorov equations whose precise formulations are difficult to locate in the literature.

\section{Preliminaries}\label{sec:preliminaries}

\subsection{Upper and lower bounds on the coefficients}

In this section, we remind the reader of a few well known bounds on the nonlocal coefficients appearing in \eqref{e.LFD}. We begin by stating the following standard bounds on the nonlocal coefficients $A$ and $\nabla a$ in \eqref{e.LFD}. Since both $A$ and $\nabla a$ are convolution-type operators appearing in the Landau equation, and $\nabla a$ is even explicitly the gradient of the inverse Laplacian, the following weighted estimates are simple, well known, and do not depend on the sign of the input.

\begin{lemma}\label{lem:coefficient_bounds}
    Suppose $g \in L^\infty_m(\R^3)$ for some $m \in (2,3) \cup(3,\infty)$. Then,
    \begin{equation*}
        \abs{A[g](v)} \lesssim \left(\brak{v}^{-1} + \brak{v}^{2-m}\right) \norm{g}_{L^\infty_m}.
    \end{equation*}
    Similarly, if $g\in L^\infty_m(\R^3)$ for some $m \in (1,3)\cup (3,\infty)$, then
    \begin{equation*}
        \abs{\nabla a[g](v)} \lesssim \left(\brak{v}^{-2} + \brak{v}^{1 - m}\right)\norm{g}_{L^\infty_m}.
    \end{equation*}
    If $m>5$, we have the improved estimate
    \begin{equation}
        A[g] \lesssim \left(\frac{\Pi(v)}{\vv} + \frac{\Id - \Pi(v)}{\vv^3}\right) \|g\|_{L^\infty_m}.
    \end{equation}
    The implicit constant above depends only on $m-5$.
\end{lemma}

\begin{proof}
All three estimates are standard convolution estimates. The last additionally uses the identity\newline $\Pi(v-w)v\cdot v = \Pi(v-w)w\cdot w$; see, e.g.,~\cite[Lemma 4.1, 4.2]{CM2017verysoft}.
\end{proof}

A key ingredient in our proof is the following connection between pointwise bounds on a function $g$ and coercivity bounds on the diffusion matrix $A[g(1-g)]$.

\begin{lemma}\label{l.pointwise_to_ellipticity}
    Suppose that $g$ satisfies, for some $r>0$, $\delta\in (0,\sfrac12)$, and $v_0\in \R^3$,
    \begin{equation}
        \delta \1_{B_r(v_0)} \leq g \leq 1 - \delta \1_{B_r(v_0)}.
    \end{equation}
    Then
    \begin{equation}
        A[g(1-g)](v) \gtrsim \frac{\Pi(v)}{\vv} + \frac{\Id - \Pi(v)}{\vv^3},
    \end{equation}
    where the implicit constant depends only on $r$, $\delta$, and $v_0$ only.
\end{lemma}
\begin{proof}
    By assumption, we have that
    \begin{equation}
        \delta(1-\delta) \1_{B_r(v_0)} \leq g(1-g).
    \end{equation}
    The conclusion then follows by applying \cite[Lemma~4.3]{HST2018landau} (see, also, \cite[Lemma~3.1]{silvestre2015landau} for an earlier version of this bound with a slightly less compatible statement).
\end{proof}

\subsection{H\"older regularity of the coefficients}

\begin{lemma}\label{l.Holder_coefficients}
    Fix $\alpha \in (0,1)$ and let $\alpha' = \sfrac{2\alpha}{3}$.  Suppose that $\vv^m g \in C^{0,\alpha}_{\rm kin}$ for some $m\neq 3$.  Then, if $m > 2 + \sfrac\alpha3$,
    \begin{equation}
        \Big[\vv^{(3-m)_+ - 1 + \frac{\alpha}{3} }A[g]\Big]_{C^{0,\alpha'}_{\rm kin}}
        \lesssim [ \vv^m g ] _{C^{0,\alpha}_{\rm kin}}.
    \end{equation}
    If $m> 1 + \sfrac\alpha3$,
    \begin{equation}
        \Big[\vv^{(3-m)_+ - 2 + \frac{\alpha}{3}} \nabla_v a[g]\Big]_{C^{0,\alpha}_{\rm kin}} \lesssim [ \vv^m g ] _{C^{0,\alpha}_{\rm kin}}.
    \end{equation}
\end{lemma}
\begin{proof}
Fix any $z_1,z_2 \in Q_1(z_0)$.  We first observe that, by a straightforward computation, 
\begin{equation}
    d(z_1 - (0,0,w), z_2 - (0,0,w))
    \lesssim d(z_1,z_2) + |w|^\frac13 d(z_1,z_2)^\frac{2}{3}
\end{equation}
for any $w\in\R^3$. 
Using this in the convolution defining $A$, we find
\begin{equation}
\begin{split}
    |A[g](z_1) - A[g](z_2)|
    &\leq
    \int_{\R^3}
        \frac{1}{|w|}
        \left(g(z_1 - (0,0,w)) - g(z'_2 - (0,0,w))\right) \dd w
    \\&
    \lesssim
    \int_{\R^3}
        \frac{1}{|w|}
        \frac{[ \vv^m g ] _{C^{0,\alpha}_{\rm kin}}}{\brak{v_0-w}^m}\left( d(z,z')^\alpha + |w|^\frac{\alpha}{3} d(z_1,z_2)^\frac{2\alpha}{3}\right)\dd w
    \\&
    \lesssim 
        \vvo^{(3-m)_+ - 1 + \frac{\alpha}{3}} \frac{[ \vv^m g ] _{C^{0,\alpha}_{\rm kin}}}d(z_1, z_2)^\frac{2\alpha}{3}. 
\end{split}
\end{equation}
Above, we used a well-known estimate on convolutions of polynomially decaying functions.  A similar argument establishes the claim for $\nabla a$.
\end{proof}

\section{Global Upper Bounds}\label{sec:upper_bounds}

In this section, we prove \Cref{thm:upper_bounds}, which contains the essential {\em a priori} upper bounds necessary for constructing strong solutions. The main estimate will be proved using an argument reminiscent of the comparison principle. Consequently, the essential step is the construction of suitable barrier functions. Throughout this section, we use a polynomial barrier $b(v)$, which includes a single scalar parameter $\Lambda > 0$ that must be tracked. In the application of \Cref{lem:barrier} to prove \Cref{thm:upper_bounds}, $\Lambda$ will be a function of time.
More precisely, $b(v)$ takes the form
\begin{equation}\label{defn:barrier}
    b(v) \coloneqq \Lambda b_1(v),\quad b_1 \in C^\infty(\R^3)\text{ is a fixed function satisfying } \begin{cases}
        \abs{v}^{-m} & \text{if } \abs{v} \ge 1\\
        \brak{v}^{-m} \le b_1(v) \le \abs{v}^{-m} &\text{if} \abs{v} < 1.\\
    \end{cases}
\end{equation}
The exact behavior of the function $b_1(v)$ for small velocities is unimportant.  The usefulness of the definition $b_1(v) = \abs{v}^{-m}$ for large velocities is that it has simpler scaling properties than the usual Japanese bracket weight $\brak{v}^{-m}$. The main property of the barrier defined in \eqref{defn:barrier} is the following estimate for the behavior at a contact point:

\begin{lemma}\label{lem:barrier}
    For any $m > 2$ and $\Lambda > 0$, define $b(v)$ via \eqref{defn:barrier}. Suppose that $f \in C^2(\R^3)$ satisfies 
    \begin{equation*}
        0 \le f \le 1, \qquad  f \le b, \qquad \text{and} \qquad f(v_0) = b(v_0) \quad \text{for some }v_0 \in \R^3.
    \end{equation*}
    Then, the collision operator satisfies 
    \begin{equation*}
        \QLFD(f) (v_0)\le K b(v_0),
    \end{equation*}
    where $K = K(m)$ is a constant depending only on $m$.
\end{lemma}

\begin{remark}[Stability under artificial viscosity]\label{rem:apriori_upper_bounds}
    The contact point estimate in \Cref{lem:barrier} remains valid if \eqref{e.LFD} is modified by adding an artificial viscosity term $\kappa \Delta_v f$, with $\kappa \ge 0$. Indeed, at a contact point with a barrier $b$, the contribution of the viscous term is bounded as $\kappa\Delta_v f \le \kappa \Delta_v b \le C(m)\kappa b$, resulting in a slightly larger constant $K + C(m)\kappa$. Thus, in the presence of artificial viscosity, the estimate in \Cref{thm:upper_bounds} holds with $K$ replaced by $K + C(m)\kappa$. This observation will be used in \Cref{sec:existence}.
\end{remark}

Before proving \Cref{lem:barrier}, we explain why it implies the propagation of decay claimed in \Cref{thm:upper_bounds}.

\subsection{Propagation of Decay}

We now discuss the precise hypotheses in \Cref{thm:upper_bounds} by specifying the exact qualitative regularity and decay assumptions imposed on the solution. 
When proving the a priori estimate for the propagation of $\norm{f(t)}_{L^\infty_m}$ for $m > 2$, we assume
\begin{equation}\label{eq:qualitative_decay}
    \sup_{0 < t < T,\,x\in \Omega,\,v\in \R^3}\left(\brak{x}^\delta + \brak{v}^{m+\delta}\right) f(t,x,v) < \infty, \qquad \text{for some }\delta > 0.
\end{equation}
Additionally, we assume that $(\partial_t + v\cdot \nabla_x)f$ and $\nabla^2_v f$ are continuous functions on $[0,T]\times \Omega\times \R^3$, and that $(\partial_t + v\cdot \nabla_x)f$ and $\QLFD(f)$ are equal pointwise everywhere (note that the qualitative decay and regularity assumptions ensure $\QLFD(f)$ is a continuous function as well).

The above assumptions---both \eqref{eq:qualitative_decay} and the regularity---are qualitative conditions used only to justify the comparison argument; neither  enters the main estimate in \Cref{thm:upper_bounds} quantitatively. As such, these assumptions could likely be removed at the expense of substantial technical arguments; however, the present formulation is sufficient for our purposes.

\begin{proof}[Proof of \Cref{thm:upper_bounds}]
Suppose $f$ is a solution to \eqref{e.LFD} on $[0,T]$ with initial data $f_{\rm in}$. Then, for any $m > 2$ and $\eps > 0$, define a time-dependent barrier function $b(t,v)$ of the form specified by \eqref{defn:barrier}:
\begin{equation}
    b(t,v)
    \coloneqq \underbrace{\left(\norm{f_{\rm in}}_{L^\infty_m} + \eps\right)\exp\left((K + \eps)t\right)}_{\coloneqq \,\Lambda(t)} b_1(v),
\end{equation}
where $K$ is the constant from \Cref{lem:barrier}.
The qualitative decay of $f$ in \eqref{eq:qualitative_decay} implies there is a constant $C_0 > 0$ so that
\begin{equation*}
    \abs{f(t,x,v)} \le \min\left(\frac{C_0}{\brak{v}^{m+\delta}},\frac{C_0}{\brak{x}^\delta}\right) \qquad \text{for each }(t,x,v)\in [0,T]\times \Omega\times \R^3.
\end{equation*}
Note that for $\Omega = \T^3$, the $\brak{x}^\delta$ term is redundant and can be omitted.
Consequently, for any contact point $(t,x,v) \in \set{f= b}$, using $b_1(v) \ge \brak{v}^{-m}$, we must have
\begin{equation*}
    \brak{v}^\delta \le \frac{C_0}{\Lambda(t)} \le \frac{C_0}{\norm{f_{\rm in}}_{L^\infty_m} + \eps} \qquad \text{and} \qquad\brak{x}^\delta \le \frac{C_0\brak{v}^m}{\Lambda(t)} \le \left(\frac{C_0}{\norm{f_{\rm in}}_{L^\infty_m} + \eps}\right)^{\frac{m}{\delta}+1}.
\end{equation*}
Combined with the continuity of $f$ and $b$, we conclude that the set of contact points $\set{f = b}$ is compact. This is the only place where the qualitative decay of $f$ is used.

Now, suppose for the sake of contradiction that $f \ge b$ somewhere in $[0,T] \times \Omega \times \R^3$. Since $f < b$ at $t=0$, by construction, the compactness of $\set{f = b}$ guarantees the existence of a first contact time $t_0 \in (0,T]$ and a point $(x_0,v_0)\in \Omega\times \R^3$ where $b(t_0,v_0) = f(t_0,x_0,v_0)$. Observe that $f(t_0) \leq b(t_0)$, by construction. Thus, we may apply \Cref{lem:barrier} to deduce that
\begin{equation}\label{e.c071301}
    \QLFD(f)(t_0,x_0,v_0) \le K b(t_0,v_0)
\end{equation}
Also, since $f-b$ attains its maximum value over $[0,t_0]\times \Omega\times \R^3$ at $(t_0,x_0,v_0)$, we have
\begin{equation}\label{e.c071302}
    (\partial_t +v_0\cdot\nabla_x) f(t_0,x_0,v_0) \ge (\partial_t +v_0\cdot\nabla_x) b(t_0,v_0).
\end{equation}
Combining~\eqref{e.c071301} and~\eqref{e.c071302} with~\eqref{e.LFD}, we find
\begin{equation*}
\begin{aligned}
        (K+\eps)b(t_0,v_0)
        &= \left(\partial_t + v_0 \cdot \nabla_x\right) b(t_0,v_0)
        \le
        \left(\partial_t + v_0 \cdot\nabla_x\right)f(t_0,x_0,v_0)
        = \QLFD(f)(t_0,x_0,v_0)
        \\&
        \le K b(t_0,v_0) ,
\end{aligned}
\end{equation*}
which is evidently a contradiction since $\eps > 0$ and $b(t_0,v_0) > 0$.
Therefore, the set of contact points is empty, and we have established the desired pointwise bound $f < b$. Letting $\eps \to 0^+$ completes the proof.
\end{proof}

This reduces the proof of \Cref{thm:upper_bounds} to establishing \Cref{lem:barrier}. The remainder of the section is devoted to the proof of \Cref{lem:barrier}.

\subsection{The Nonlocal Contributions}

The coefficient bounds from \Cref{lem:coefficient_bounds} are too imprecise to deduce \Cref{lem:barrier} for the full range $m > 2$; in particular, they fail to close once $m \ge 3$. Note that because the kernels of $A$ and $\nabla a$ are supported everywhere and decay slowly at large velocities, it is impossible to obtain stronger decay for $A[f]$ or $\nabla a[f]$ by taking $f$ to have faster decay.  Indeed, the decay seen in \Cref{lem:coefficient_bounds} does not improve for $m > 3$.

Instead, to obtain more precise bounds, we show that the nonlocal contribution preventing faster decay has a good sign and can be discarded. The role of the following lemma is to produce this sign by analyzing the integrands appearing in the nonlocal terms:

\begin{lemma}\label{lem:geometry}
    Fix $\Lambda > 0$ and $m > 0$ and define $b(v)$ as in \eqref{defn:barrier} using these values. Then, there is a $\delta\in (0,1)$, depending only on $m$, such that
    \begin{equation}\label{eq:geometry}
        (v-w) \cdot \nabla b(v) \le 0
        \qquad\text{and}\qquad
        \Pi(v-w):\nabla^2 b(v) \le 0
    \end{equation}
    for all $|v| \geq 1$ and $|w|\leq \delta |v|$.
\end{lemma}

\begin{proof}
We begin by computing derivatives of $b$ for $\abs{v} \ge 1$:
\begin{equation*}
    \nabla b = -m \left(\frac{v}{\abs{v}^2} \right) b \qquad \text{and} \qquad \nabla^2 b = \frac{m}{\abs{v}^2}\left((m+2)\frac{v\tens v}{\abs{v}^2} - \Id \right) b.
\end{equation*}
We now prove the first and second derivative estimates separately, reducing each by scaling to a statement that is evidently true. Substituting the form of $\nabla b$ into the claimed first derivative estimate, we must show that
\begin{equation}\label{eq:claimed_estimate1}
    -m\left((v-w) \cdot \frac{v}{\abs{v}^2} \right) b \le 0, \qquad \qquad \text{whenever }w \in B_{\delta\abs{v}}.
\end{equation}
Using that $m,\, b\ge 0$, \eqref{eq:claimed_estimate1} reduces to the claim
\begin{equation}\label{eq:reduced_estimate1}
    (v-w) \cdot v \ge 0, \qquad \qquad \text{whenever }w \in B_{\delta\abs{v}}.
\end{equation}
Scaling both $v$ and $w$ identically and applying rotation invariance reduces \eqref{eq:reduced_estimate1} to 
\begin{equation}\label{eq:scaled_estimate1}
    (\hat e-w) \cdot \hat e \ge 0, \qquad \qquad \text{whenever }w \in B_{\delta}\; \text{for a fixed unit vector  }\hat e\in\R^3.
\end{equation}
The estimate \eqref{eq:scaled_estimate1} is evidently true for any $\delta \in (0,1)$ by the Cauchy-Schwarz inequality or by a perturbation of $w = 0$. Putting together all of the above, we have shown the first claim in~\eqref{eq:geometry}.

Next, substituting the form of $\nabla^2 b$ into the claimed second derivative estimate, we must show that
\begin{equation}\label{eq:claimed_estimate2}
    \Pi(v-w):\frac{m}{\abs{v}^2}\left((m+2)\frac{v\tens v}{\abs{v}^2} - \Id \right) b \le 0, \qquad \qquad \text{whenever }w \in B_{\delta\abs{v}}.
\end{equation}
Using that $\Pi(v-w)$ is a rank two orthogonal projection, \eqref{eq:claimed_estimate2} simplifies to
\begin{equation}\label{eq:reduced_estimate2}
    \left((m+2)\frac{\Pi(v-w)v\cdot v}{\abs{v}^2} - 2 \right) \le 0 \qquad \qquad \text{whenever }w \in B_{\delta\abs{v}}.
\end{equation}
Again scaling $v$ and $w$ identically and applying rotation invariance reduces \eqref{eq:reduced_estimate2} to a  simplified statement:
\begin{equation}\label{eq:scaled_estimate2}
    (m+2)\Pi(\hat e-w)\hat e\cdot \hat e - 2  \le 0
    \qquad \qquad \text{whenever }w \in B_{\delta}\; \text{for a fixed unit vector  }\hat e\in\R^3.
\end{equation}
This is evidently true for $w = 0$, where the left hand side equals $-2$ (recall that $\Pi(\hat e)\hat e=0$). By continuity, there exists a $\delta \in (0,1)$ for which~\eqref{eq:scaled_estimate2} holds. This concludes the proof.
\end{proof}

We are now prepared to prove \Cref{lem:barrier}.

\subsection{The Contact Point Estimate}

\begin{proof}[Proof of \Cref{lem:barrier}]
    Writing $\QLFD(f)$ in nondivergence form, we must find a constant $K$, depending only on $m$, for which    \begin{equation}\label{eq:claimed_bound}
        \QLFD(f)(v_0) = A[f(1-f)]:\nabla^2 f(v_0) + \left(2f\nabla a[f] - \nabla a[f^2]\right)\cdot \nabla f(v_0) + (1-f)f^2(v_0) \le Kb(v_0).
    \end{equation}
    We begin with the zeroth-order term on the left hand side of~\eqref{eq:claimed_bound}.  This term is purely local, so, using $f(v_0) = b(v_0)$ and the Pauli exclusion bound, we obtain
    \begin{equation}
        (1-f)f^2(v_0) = b(v_0)(1-f)f \le b(v_0).
    \end{equation}
    This is consistent with our goal~\eqref{eq:claimed_bound}.
    
    Next, we notice that we may replace the derivatives of $f$ with $b$.  Indeed, $f - b \in C^2(\R^3)$ attains its maximum value at $v_0$, which yields
    \begin{equation}
        \nabla f(v_0) = \nabla b(v_0) \qquad \text{and} \qquad \nabla^2 f(v_0) \le \nabla^2 b(v_0).
    \end{equation}
    These two observations reduce \eqref{eq:claimed_bound} to establishing:
    \begin{equation}\label{eq:claimed_bound2}
        \underbrace{A[f(1-f)]:\nabla^2 b(v_0)}_{\rm Term\; I} + \underbrace{\left(2f\nabla a[f] - \nabla a[f^2]\right)\cdot \nabla b(v_0)}_{\rm Term \; II}  \le Kb(v_0),
    \end{equation}
    for some constant $K$ depending only on $m$.
    
    We consider first the simple case when $|v_0| \leq 1$.  Here, the precise decay at large velocities is irrelevant, and we estimate crudely using the coefficient bounds in \Cref{lem:coefficient_bounds}, monotonicity of weighted norms, and $\abs{\nabla b(v_0)} + \abs{\nabla^2 b(v_0)} \lesssim_m b(v_0)$ to obtain
    \begin{equation*}
        A[f(1-f)]:\nabla^2 b(v_0) + \left(2f\nabla a[f] - \nabla a[f^2]\right)\cdot \nabla b(v_0) \lesssim  b(v_0) \norm{f}_{L^\infty_m}, \qquad \text{when }\abs{v_0} \le 1.
    \end{equation*}
    Additionally, $f\le b$ and $v_0$ is a contact point meaning $b(v_0) = f(v_0) \le 1$ and we conclude
    \begin{equation*}
        \mathrm{Term \; I} + \mathrm{Term\;II}\lesssim \norm{b}_{L^\infty_m} \lesssim \Lambda \lesssim b(v_0), \qquad \text{when }\abs{v_0} \le 1.
    \end{equation*}

    We now consider the main case when $|v_0| \geq 1$, which will complete the proof. 
    We start by showing \eqref{eq:claimed_bound2} for the second-order term, $\mathrm{Term \; I}$. Using the sign provided by \Cref{lem:geometry} yields
    \begin{equation*}
    \begin{aligned}
        {\rm Term\; I} = \frac{1}{8\pi}\int_{\R^3} &\left[\frac{\Pi(v_0-w):\nabla^2 b(v_0)}{\abs{v_0 - w}}\right]f(w)(1-f(w)) \dd w\\
            &\le \frac{1}{8\pi} \int_{\R^3 \setminus B_{\delta\abs{v_0}}} \left[\frac{\Pi(v_0-w):\nabla^2 b(v_0)}{\abs{v_0 - w}}\right]f(w)(1-f(w)) \dd w.
    \end{aligned}
    \end{equation*}
    Then, we use that
    \begin{equation}\label{e.c071303}
        b(v_0) = f(v_0) \leq 1,
        \quad
        \abs{\nabla^2 b(v_0)} \lesssim \frac{b(v_0)}{\abs{v_0}^{2}}, \quad\text{ and }\quad
        f(w) \le b(w) \le \frac{\Lambda}{\abs{w}^m} \text{ for all }w \in \R^3.
    \end{equation}
    to deduce the bound
    \begin{equation}
    \begin{split}
        {\rm Term\; I} &\lesssim \frac{b(v_0)}{\abs{v_0}^2}\int_{\R^3 \setminus B_{\delta\abs{v_0}}} \frac{b(w)}{\abs{v_0 - w}} \dd w
        \le \frac{\Lambda}{\abs{v_0}^2} \int_{\R^3\setminus B_{\delta\abs{v_0}}} \frac{1}{\abs{v_0 - w}\abs{w}^{m}} \dd w.
    \end{split}
    \end{equation}
    The decay of the last integral is computed by a simple change of variables:
    \begin{equation}
        {\rm Term\; I}  \lesssim  \frac{\Lambda}{\abs{v_0}^m} \int_{\R^3\setminus B_{\delta}} \frac{1}{\abs{\hat{v}_0 - w}\abs{w}^{m}} \dd w \lesssim b(v_0).
    \end{equation}
    The implicit constant above depends only on $m$ (recall that $\delta$ depends only on $m$). This completes the verification of \eqref{eq:claimed_bound2} for $\mathrm{Term \; I}$.

    We conclude by showing \eqref{eq:claimed_bound2} for the first-order term, $\mathrm{Term \; II}$.
    We expand $\mathrm{Term \; II}$ as a nonlocal term, decomposing it into two integrals, and use the sign provided by \Cref{lem:geometry} to conclude:
    \begin{equation}
    \begin{aligned}
        \mathrm{Term\; II} &= \frac{1}{4\pi}\int_{\R^3} \left[f^2(w) - 2f(v_0)f(w)\right] \frac{(v_0-w) \cdot \nabla b(v_0)}{\abs{v_0 - w}^3} \dd w \\
            &= \frac{1}{4\pi}\int_{\R^3\setminus B_{\delta\abs{v_0}}} \left[f^2(w) - 2f(v_0)f(w)\right]  \frac{(v_0-w) \cdot \nabla b(v_0)}{\abs{v_0 - w}^3} \dd w\\
                &\qquad\qquad\qquad\qquad\qquad+  \frac{1}{4\pi}\int_{B_{\delta\abs{v_0}}} \left[\left(f(v_0) - f(w)\right)^2 - f(v_0)^2\right] \frac{(v_0 - w)\cdot \nabla b(v_0)}{\abs{v_0-w}^3} \dd w\\
            &\lesssim 
            \abs{\nabla b(v_0)}\left(\int_{\R^3\setminus B_{\delta\abs{v_0}}}\frac{f(w)}{\abs{v_0 - w}^2} \dd w +  \int_{B_{\delta\abs{v_0}}} \frac{f(v_0)^2}{\abs{v_0-w}^2} \dd w\right),
    \end{aligned}
    \end{equation}
    Recall from~\eqref{e.c071303} that $f(w) \leq b(w)$ and $f(v_0) \leq 1$.  Hence,
    \begin{equation}
    \begin{aligned}
        {\rm Term\; II} &\lesssim \frac{b(v_0)}{\abs{v_0}}\int_{\R^3\setminus B_{\delta\abs{v_0}}}\frac{b(w)}{\abs{v_0 - w}^2} \dd w + \frac{b(v_0)}{\abs{v_0}}\int_{B_{\delta\abs{v_0}}} \frac{1}{\abs{v_0-w}^2} \dd w\\
        &\le \frac{1}{\abs{v_0}} \int_{\R^3\setminus B_{\delta\abs{v_0}}} \frac{\Lambda}{\abs{v_0 - w}^2\abs{w}^{m}} \dd w + \frac{b(v_0)}{\abs{v_0}} \int_{B_{\delta\abs{v_0}}} \frac{1}{\abs{v_0 - w}^2} \dd w.
    \end{aligned}
    \end{equation}
    The implicit constant above, coming only from the computation of $\nabla b$, depends only on $m$. 
    The decay of the integrals on the last line is computed by a simple change of variables:
    \begin{equation}
        {\rm Term\; II}  \lesssim  \frac{\Lambda}{\abs{v_0}^m} \int_{\R^3\setminus B_{\delta}} \frac{1}{\abs{\hat{v}_0 - w}^2\abs{w}^{m}} \dd w + b(v_0) \int_{B_{\delta}} \frac{1}{\abs{\hat{v}_0 - w}^2} \dd w \approx \frac{\Lambda}{|v_0|^m} + b(v_0) = 2 b(v_0).
    \end{equation}
    Again, the implicit constant depends only on $\delta$ and $m$. This completes the verification of \eqref{eq:claimed_bound2} for $\mathrm{Term \; II}$. As discussed above, the bound \eqref{eq:claimed_bound} holds, and the proof is complete.
\end{proof}

\section{Mass Spreading and Desaturation}\label{sec:nonvacuum}

We deduce \Cref{t.nonvacuum} from a slightly stronger result that explicitly includes an artificial viscosity term.
\begin{proposition}\label{p.nonvacuum}
    Fix $\kappa \in[0,\kappa_0]$. Suppose that the assumptions of \Cref{t.nonvacuum} hold, with the only modification being that we assume that $f$ is a classical solution to
    \begin{equation}\label{e.viscous_LFD}
        (\partial_t + v\cdot\nabla_x)f
            = \kappa \Delta_v f + \QLFD(f),
    \end{equation}
    in place of~\eqref{e.LFD}.  Then the conclusion~\eqref{e.nonvacuum_on_compacts} holds with additional dependence on $\kappa_0$.
\end{proposition}

We comment briefly on the relationship between \Cref{t.nonvacuum} and \Cref{p.nonvacuum}.  Clearly, the main estimate~\eqref{e.nonvacuum_on_compacts} of \Cref{t.nonvacuum} follows from \Cref{p.nonvacuum} after taking $\kappa = 0$. The ellipticity bound~\eqref{e.coercivity} in \Cref{t.nonvacuum} then follows directly from \Cref{l.pointwise_to_ellipticity}. The final claim in \Cref{t.nonvacuum} is that the bound~\eqref{e.coercivity} holds uniformly up to $t_0 = 0$ when $R < r$. This comes from a detailed reading of the proof; for $(x,v)$ in the initial non-vacuum region, the relevant constants in ``Step 1: Sustaining an ellipticity core'' below do not degenerate at the initial time.

The remainder of this section is organized as follows. First, in \Cref{ss.nonvacuum_heuristics} we give a heuristic idea of the proof, along with a discussion of the technical issues in constructing a proof following this heuristic. This introduces the main two actions for spreading non-vacuum/non-saturation: transport in $x$ and diffusion in $v$. We state two general lemmas covering this behavior in \Cref{ss.mass_spreading_lemmas}. Then we show how to combine them to prove \Cref{p.nonvacuum} in \Cref{ss.nonvacuum_proof}.

\subsection{Main ideas}\label{ss.nonvacuum_heuristics}

The proof of \Cref{p.nonvacuum} is based on a simple heuristic related to the kinetic dynamics and the nonlinear structure of~\eqref{e.viscous_LFD}. This picture is somewhat obscured by the formal proof, which largely amounts to introducing and balancing a number of technical parameters. To improve readability, we first give a high-level overview of the argument and introduce most of the relevant parameters. Their precise choices and compatibility conditions are deferred to the proof.

\subsubsection*{\underline{Overview}}

Without loss of generality, assume that the non-vacuum condition \eqref{d.nonvacuum} is centered at the origin in phase-space, so that $(x_0,v_0) = (0,0)$. Fix an arbitrary endpoint $(x_e,v_e)$ and arbitrary time $t_e > 0$. Our objective is to propagate a quantitative lower bound from a neighborhood of $(0,0)$ to a neighborhood of $(x_e,v_e)$ by time $t_e$. 

For kinetic equations, the interplay between diffusion in velocity and transport in space acts as a substitute for diffusion in all phase-space variables, a mechanism known as hypoellipticity. In the present setting, this suggests one possible path through phase space, dividing the available time into three stages: First, diffusion in velocity moves the lower bound from $(0,0)$ to a neighborhood of $\left(0,\sfrac{2x_e}{t_e}\right)$. Second, it is transported in space at the intermediate velocity $\sfrac{2x_e}{t_e}$ for a duration $t_e/2$, carrying it to the desired spatial coordinate. Third, diffusion in velocity carries the lower bound to the desired endpoint $(x_e,v_e)$. Schematically,
\begin{equation*}
(0,0)\xrightarrow{\text{velocity diffusion}}\left(0,\sfrac{2x_e}{t_e}\right)\xrightarrow[t_e/2]{\text{spatial transport}}\left(x_e,\sfrac{2x_e}{t_e}\right)\xrightarrow{\text{velocity diffusion}}(x_e,v_e).
\end{equation*}
We have allocated half the available time to transport, but we purposefully have not specified the time for diffusion since it dominates on a significantly shorter time scale. A more faithful schematic must account for the shorter time scales allotted to the diffusion stages. We denote by $[s_1,s_2]$ the time interval for the first diffusion with duration $\Delta s$ and similarly $[\sigma_1,\sigma_2]$ and $\Delta \sigma$ for the second diffusion. Any remaining time can be absorbed into an initial waiting period; despite looking phenomenologically different than transport/diffusion stages, this delay can be interpreted as a transport stage since we are working in the rest frame. The resulting detailed schematic is
\begin{equation*}
    (0,0) \xrightarrow[s_1]{\text{delay}} (0,0)\xrightarrow[\Delta s]{\text{first diffusion}} (0,\sfrac{2x_e}{t_e}) \xrightarrow[t_e/2]{\text{transport}} (x_e,\sfrac{2x_e}{t_e}) \xrightarrow[\Delta \sigma]{\text{second diffusion}} (x_e,v_e).
\end{equation*}
A phase-space diagram of this route is depicted in \Cref{f.mass_spreading}. These four stages correspond, roughly, to Steps 1, 2, 3, and 4 of the proof below. 

Although the delay appears to be a technical contrivance, this is the fundamental basis for the nonlinear argument. Indeed, in nonlinear diffusion models, any ellipticity must be generated by the solution itself, which can restrict spreading (e.g., vacuum persistence in the porous medium equation). The delay \emph{sustains} lower bound on $f(1-f)$ around the origin.  The nonlocal diffusion \emph{spreads} coercivity in velocity prior to and independently of the spreading of mass during a diffusion stage. We call this region an \emph{ellipticity core}, which functions as the source of all lower bounds generated at future steps.

\fmassspreading

\subsubsection*{\underline{Main lemmas}}

The formal tools for understanding the two fundamental behaviors above, transport and diffusion, are two linear estimates for nonnegative supersolutions of generic kinetic equations---\Cref{l.transport} and \Cref{l.spreading} below---which propagate lower bounds along certain trajectories during transport stages and diffusion stages, respectively. At each stage, we apply the appropriate estimate to both $f$ and $1-f$, thereby propagating lower bounds for both quantities simultaneously. Although superficially similar, the two estimates play different roles: crucially, sustaining a lower bound in a transport stage does not require coercivity, whereas spreading lower bounds in a diffusion stage requires quantitative coercivity.

\subsubsection*{\underline{Complications}} 

Of course, the two propagation lemmas do not give a literal time-splitting of the equation. In each application of the lemmas, we can either \emph{spread} or \emph{sustain} lower bounds along particular type of path through phase space compatible with the hypoelliptic structure of~\eqref{e.LFD}. We call these paths $(x_t,v_t)$ \emph{kinetic trajectories}, which must solve the following compatibility condition:
\begin{equation}\label{e.characteristic}
    \dot{x}_t = v_t.
\end{equation}
In implementing the heuristic procedure described above, we encounter three main difficulties:
\begin{enumerate}
    \item \emph{Diffusion during transport.} During a transport stage, diffusion and drift in $v$ remain active and may degrade the lower bound being transported.

    \item \emph{Transport during diffusion.} During a diffusion stage, the compatibility condition \eqref{e.characteristic} forces the spatial component \(x_t\) of the kinetic trajectory to move. Consequently, the trajectory may exit the spatial region where the diffusion step is intended to take place.
    
    \item \emph{Ellipticity during diffusion.} Applying the diffusion lemma requires a quantitative lower bound for the diffusion coefficient $A[f(1-f)]$ along a moving neighborhood of $(x_t,v_t)$.
\end{enumerate}
Both difficulties (i) and (ii) are already present in the linear, uniformly elliptic setting ($\kappa > 0$) and can be dealt by choosing time scales appropriately. Regarding (i), the upper bounds of \Cref{thm:upper_bounds} control unwanted transport, diffusion, and drift effects, so they can be made harmless by working on a perturbative time scale $t_{\rm core} \ll 1$. Reducing to the case $t_e \le t_{\rm core}$ fixes a time scale for the transport stage. Regarding (ii), transport only dominates on long time scales; picking the diffusion times scales $\Delta s \ll t_e$ and $\Delta \sigma \ll t_e$ substantially smaller makes transport effects negligible. 
Finally, (iii) is the essential nonlinear difficulty. 

\subsection{The diffusion and transport lemmas}\label{ss.mass_spreading_lemmas}

We now introduce the two lemmas that formalize the heuristic notions of \emph{retaining} and \emph{spreading} mass along kinetic trajectories. These lemmas are generic, reuseable lower bounds for a general class of linear kinetic equations problem. The main work is in combining them appropriately to address the nonlinearity in LFD. More precisely, we will consider nonnegative, classical supersolutions of the non-divergence form equation:
\begin{equation}\label{e.kfp_lower}
	\begin{cases}
		(\partial_t + v\cdot\nabla_x) g \geq \Tr(\mathcal A D^2_v g) + \mathcal B\cdot \nabla_v g + \mathcal C g \qquad&\text{ in } (0,T)\times \R^3\times \R^3,\\
		g = g_{\rm in} \qquad &\text{ on } \{0\} \times \R^3\times \R^3.
	\end{cases}
\end{equation}
For simplicity, we assume that all relevant terms, e.g. $(\partial_t + v\cdot \nabla_x)g$, $g$, $\nabla_v^2 g$, $\cA$, $\cB$, and $\cC$ are continuous and the inequality \eqref{e.kfp_lower} holds pointwise everywhere. Additionally, we make the structural assumption that
\begin{equation}\label{e.ellipticity}
    0\le \cA(t,x,v) \quad \text{is symmetric, positive semidefinite with ellipticity }\quad \frac{1}{\lambda(t,x,v)} \coloneqq  \inf_{\abs{\xi} = 1} \cA(t,x,v)\xi \cdot \xi.
\end{equation}
Note that we assume only $0 < \lambda(t,x,v) \le \infty$; the case $\lambda(t,x,v) = \infty$ may occur wherever $\cA(t,x,v)$ is degenerate.

As discussed above, the understanding the geometry associated to \eqref{e.kfp_lower} is crucial. We introduce the following notation. Associated to a time interval $0 < \tau \le T$, initial position $x_0 \in \R^3$, initial velocity $v_0\in \R^3$, final/target velocity $v_\tau \in \R^3$ we may solve the ODE system \eqref{e.characteristic} to define a kinetic trajectory:
\begin{equation}\label{e.kinetic_trajectory}
    v_t = v_0 + \frac{t(v_\tau - v_0)}{\tau} \qquad \text{and} \qquad x_t = x_0 + \int_0^t v_s \dd s.
\end{equation}
Additionally, for a fixed radius $r > 0$, we define a fattened version of $(x_t,v_t)$ and associated localized norms of the the coefficients:
\begin{equation}\label{e.path}
\begin{aligned}
    \cP_t^r &\coloneqq B_r(x_t)\times B_r(v_t) = \set{(x,v) \colon \abs{x-x_t} < r \text{ and } \abs{v-v_t} < r}\\[5pt]
        \norm{F}_{\rm path} &\coloneqq \sup_{0<t < \tau}\norm{F(t)}_{L^\infty(\cP_t^r)}, \qquad\text{where} \quad F = \cA,\,\cB,\,\cC,\text{ or }\lambda.
\end{aligned}
\end{equation}
We now prepared to state our two lemmas. They are significantly simplified from the general versions in~\cite{HSTT_vacuum}. We start with the diffusion lemma, which requires a local bound on the ellipticity of $\cA$ near the trajectory; that is, $\norm{\lambda}_{\rm path} < \infty$.

\begin{lemma}[\hspace{-.01em}{\cite[Proposition~2.2]{HSTT_vacuum}}]\label{l.spreading}
	Fix any choice of parameters $\tau, \delta, r > 0$ and $x_0, v_0, v_\tau \in \R^3$ and define the associated trajectory $(x_t,v_t)$ by~\eqref{e.kinetic_trajectory} and path by \eqref{e.path}. Suppose that $g$ is a nonnegative, classical supersolution to~\eqref{e.kfp_lower} and that
	\begin{equation}
		g_{\rm in}(x,v) > \delta \qquad \text{for each }(x,v) \in B_r(x_0)\times B_r(v_0) \qquad \text{and} \qquad \norm{\lambda}_{\rm path} < \infty.
	\end{equation}
    Then, there exists a universal constant $\mu > 0$ such that whenever
    \begin{equation}
        \tau + \frac{\tau \norm{A}_{\rm path}}{r^2} \leq \mu,
    \end{equation}
    there holds
	\begin{equation}
		g\left(\tau ,x,v\right) > \frac{\delta}{C_{\rm diff}} \qquad \text{for each }(x,v)\in B_{\sfrac{r}{4}}(x_{\tau}) \times B_{\sfrac{r}{4}}(v_\tau).
	\end{equation}
    The constant $C_{\rm diff} > 0$ depends on upper bounds for $\tau$, $|v_0 - v_\tau|$, $\|\lambda\|_{\rm path}$, $\|\mathcal B\|_{\rm path}$, and $\|\cC_-\|_{\rm path}$ as well as lower bounds for $\tau$ and $r$.
\end{lemma}

Now we state the transport lemma, the special case $v_\tau = v_0$.  This result does not use ellipticity; that is, we allow $\lambda = +\infty$. 

\begin{lemma}[\hspace{-.01em}{\cite[Proposition~2.3]{HSTT_vacuum}}]\label{l.transport}
	Fix any choice of parameters $\tau, \delta, r > 0$ and $x_0, v_0\in \R^3$ and define the associated trajectory $(x_t,v_0)$ by~\eqref{e.kinetic_trajectory} and path by \eqref{e.path}. Suppose that $g$ is a nonnegative, classical supersolution to~\eqref{e.kfp_lower} and that
	\begin{equation}
		g_{\rm in}(x,v) > \delta \qquad \text{for each }(x,v) \in B_r(x_0)\times B_r(v_0).
	\end{equation}
    Then, there exists a universal constant $\mu > 0$ such that whenever
    \begin{equation}
        \tau + \frac{\tau \norm{\cA}_{\rm path}}{r^2} + \frac{\tau \norm{\cB}_{\rm path}}{r} \leq \mu,
    \end{equation}
    there holds
	\begin{equation}
		g\left(\tau,x,v\right) > \frac{\delta}{C_{\rm tran}} \qquad \text{for each }(x,v) \in B_{\sfrac{r}{4}}(x_{\tau}) \times B_{\sfrac{r}{4}}(v_0).
	\end{equation}
    The constant $C_{\rm tran}> 0$ depends on only on upper bounds for $\tau$ and $\norm{\cC_-}_{\rm path}$.
\end{lemma}

It is possible to track the explicit dependencies of $C_{\rm diff}$ and $C_{\rm tran}$ on the parameters. This is necessary to obtain the more precise exponential lower bounds derived from mass spreading; see \cite{HSTT_vacuum} for details.

\subsection{Mass Spreading for the Landau--Fermi--Dirac equation}\label{ss.nonvacuum_proof}

Throughout the proof, it is understood that every implicit constant may depend on $v_0$, $\kappa_0$, $r$, $\delta$, $T$, $m$, and $\|f_{\rm in}\|_{L^\infty_m}$.  When a bound depends on more than those quantities, we clarify it explicitly.

\begin{proof}[Proof of \Cref{p.nonvacuum}]
We start by making some reductions. First, the proof is inherently local, so we consider only the case $\Omega = \R^3$; the case $\Omega = \T^3$ follows from considering the periodic extension of $f$. Second, we take advantage of the Galilean invariance of \eqref{e.viscous_LFD}. By performing a spatial translation and working in a moving frame, we may assume $(x_0,v_0) = (0,0)$. Third, by scaling considerations we may assume that $r=1$; otherwise, we work with
\begin{equation}
    f_r(t,x,v) = f(t,rx,rv),
\end{equation}
which solves~\eqref{e.viscous_LFD} with artificial viscosity $\frac{\kappa}{r^2}\Delta_v f$. This is fine since all bounds already depend on $\kappa_0$.

We now fix an endpoint
\begin{equation}
    (t_e,x_e,v_e) \in [t_0,T] \times B_R \times B_R.
\end{equation}
The proof is complete if we obtain lower bounds on $f$ and $1-f$ at $(t_e,x_e,v_e)$, depending only on $t_0$, $T$, and $R$ and the parameters discussed above. We now begin the proof in earnest.

\medskip
\noindent {\bf Step 0: Preliminary observations and the perturbative time scale.} 
Our first and most important observation is that the quantity 
\begin{equation}
    h\coloneqq 1 - f    
\end{equation}
satisfies a non-divergence form kinetic equation similar equation to~\eqref{e.viscous_LFD}. Indeed, we find the pair $(f,h)$ satisfy the equations
\begin{equation}\label{eq:both_equations}
\begin{split}
    (\partial_t + v\cdot\nabla_x) h &= \cA:D^2_v h + \cB \cdot \nabla_v h + \cC_hh \qquad \text{with} \qquad h(0) = h_{\rm in} \coloneqq 1-f_{\rm in}\\
    (\partial_t + v\cdot\nabla_x) f &= \cA:D^2_v f + \cB \cdot \nabla_v f + \cC_f f \qquad \text{with} \qquad f(0) = f_{\rm in}
\end{split}
\end{equation}
where the coefficients are computed as
\begin{equation}
    \cA\coloneqq \kappa\Id + A[fh], \qquad \cB = \nabla a[fh] - (h-f)\nabla a[f], \qquad \cC_h = -f^2, \qquad \text{and} \qquad \cC_f = fh.
\end{equation}
Note that Pauli exclusion gives $\norm{\cC_h}_{L^\infty} \le 1$ and $\norm{\cC_f}_{L^\infty} \le 1$, whereas combining \Cref{thm:upper_bounds} and \Cref{lem:coefficient_bounds} yields global upper bounds for the coefficients:
\begin{equation}\label{e.coefficient_bounds}
    \|\cA\|_{L^\infty} \le \Lambda \qquad \text{and} \qquad \|\cB\|_{L^\infty} \le \Lambda,
\end{equation}
where $\Lambda$ depends only on $T$, $m$, $\kappa_0$, and $\|f_{\rm in}\|_{L^\infty_m}$.
Additionally, the non-vacuum condition in \Cref{d.nonvacuum} translates to the lower bound
\begin{equation}
    \delta < h_{\rm in}, \, f_{\rm in} \qquad \text{in }B_1\times B_1.
\end{equation}
We now fix the short time scale $t_{\rm core}$ that allows us to treat unwanted diffusive and transport effects perturbatively. More precisely, for $\mu$ as in \Cref{l.spreading,l.transport} and $r_{\rm min} \coloneqq 2^{-9}$ the minimum radius at which we will apply these lemmas, we choose
\begin{equation}
    t_{\rm core}\left(1 + \frac{\Lambda}{r_{\rm min}^2} + \frac{\Lambda}{r_{\rm min}}\right) < \mu,
\end{equation}
so that the smallness conditions in \Cref{l.spreading} and \Cref{l.transport} are satisfied for all subsequent applications to both $f$ and $h$ on any subinterval of $[0,t_{\rm core}]$. 

Before proceeding, we assume without loss of generality that $t_e < T < t_{\rm core}$. Otherwise, given that the $t_{\rm core}$ is chosen depending only on the constant $\Lambda$, the argument may be iterated on time intervals of uniform length $t_{\rm core}$ to reach the time $T$. 

\medskip

\noindent {\bf Step 1: Sustaining an ellipticity core.}

As mentioned above, $t_{\rm core}\leq 1$ is sufficiently small to apply \Cref{l.transport} to both equations \eqref{eq:both_equations} at the origin to find that
\begin{equation}\label{e.mass_core}
    f(t,x,v),\, h(t,x,v) \gtrsim 1 \qquad\text{ for all } (t,x,v) \in [0,t_{\rm core}]\times B_{\frac14}\times B_{\frac14}.
\end{equation}
The nonlocality of $A$ then spreads the lower bound to an ellipticity bound globally in $v$. More precisely, discarding the artificial viscosity and applying \Cref{l.pointwise_to_ellipticity} yields
\begin{equation}\label{e.A[fh]_lower}
    \cA(t,x,v) \ge A[fh](t,x,v) \gtrsim \frac{1}{\vv^3} \qquad \text{for all }(t,x,v) \in [0,t_{\rm core}] \times B_{\frac14}\times \R^3,
\end{equation}
where the implicit constant is evidently independent of the artificial viscosity $\kappa$.
This coercivity allows us, in the next step, to spread lower bounds in \eqref{e.mass_core} on $f$ and $h$ to any velocity. 

\medskip

\noindent {\bf Step 2: Spreading to an intermediate velocity by diffusion.}

Next, we spread the lower bounds \eqref{e.mass_core} to the intermediate velocity $\sfrac{2x_e}{t_e}$ using diffusion. As discussed in the heuristics, we spread over a time interval $[s_1,s_2]$ from $(x_{s_1},v_{s_1}) = (0,0)$ to $(x_{s_2},v_{s_2})$ with $v_{s_2} = \sfrac{2x_e}{t_e}$. For the moment, $s_1$ and $s_2$ are arbitrary times satisfying $[s_1,s_2]\subset [0,t_e]$; they will be chosen below. The corresponding trajectory $(x_s,v_s)$ is then given by
\begin{equation}
    x_s = \int_{s_1}^{s} v_t \dd t \quad\text{ and }\quad v_s = \left(\frac{2 x_e}{t_e}\right)\left(\frac{s-s_1}{s_2-s_1}\right) \qquad \text{for }s_1 \le s \le s_2.
\end{equation}
After a time translation by $s_1$, we aim to apply \Cref{l.spreading} to this trajectory with $\tau = s_2 - s_1$, $r = \sfrac18$, and $\delta$ defined by \eqref{e.mass_core}. 

The initial lower bound is provided \eqref{e.mass_core} and the necessary smallness is provided by the choice of $t_{\rm core}$. Thus, to apply \Cref{l.spreading}, we need only verify that $\norm{\lambda}_{\rm path} < \infty$. 
Further, by \eqref{e.A[fh]_lower}, it suffices to show that $B_{\sfrac18}(x_s) \subset B_{\sfrac14}$ for each $s_1 < s < s_2$. We claim that this holds provided $s_1$ and $s_2$ are further required to satisfy the condition
\begin{equation}\label{e.small_time_interval}
     s_2 - s_1 \leq \frac{t_e}{64 \brak{x_e}}.
 \end{equation}
Indeed, assuming \eqref{e.small_time_interval}, we compute
\begin{equation}
    |x_s| \leq \int_{s_1}^{s_2} \left|\frac{2 x_e}{t_e}\right| \left(\frac{t-s_1}{s_2-s_1}\right)\dd t = (s_2-s_1) \frac{|x_e|}{t_e} \leq \left(\frac{t_e}{64 \langle x_e\rangle}\right)\left( \frac{|x_e|}{t_e}\right) < \frac{1}{64} \qquad \text{for each }s_1 < s < s_2.
\end{equation}
Consequently, $B_{\sfrac18}(x_s) \subset B_{\sfrac14}$, \Cref{l.spreading} applies, and we obtain the lower bound
\begin{equation}
    f(s_2,x,v),\,h(s_2,x,v) \gtrsim C_{\rm diff} \qquad\text{ for all }(x,v) \in B_{\frac{1}{32}}(x_{s_2})\times B_{\frac{1}{32}}\left(\frac{2x_e}{t_e}\right).
\end{equation}
In the present setting, the constant $C_{\rm diff}$ can only degenerate as $\abs{s_1 - s_2} \to 0^+$. Thus, fixing $s_2 - s_1$ to saturate \eqref{e.small_time_interval} and noting $B_{\frac{1}{32}}(x_{s_2}) \supset B_{\frac{1}{64}}$ yields the uniform bound
\begin{equation}\label{e.s2_uniform}
    f(s_2,x,v),\,h(s_2,x,v) \gtrsim 1 \qquad\text{ for all }(x,v) \in B_{\frac{1}{64}}\times B_{\frac{1}{64}}\left(\frac{2x_e}{t_e}\right),
\end{equation}
for any admissible choices of $s_2$. We now vary $s_2$ over the range $[\sfrac{t_e}{4},t_e]$, which is admissible because
\begin{equation}
    s_1 = s_2 - \frac{t_e}{64\brak{x_e}} \ge \frac{t_e}{4} - \frac{t_e}{64} > 0.
\end{equation}
In summary, the effect of the waiting stage $s_1$ and consequent staggered arrival times $s_2$ is a more explicit version of \eqref{e.s2_uniform}:
\begin{equation}\label{e.s2_explicit}
    f(t,x,v),\,h(t,x,v) \gtrsim 1 \qquad\text{ for all }(t,x,v) \in \left[\frac{t_e}{4},t_e\right]\times B_{\frac{1}{64}}\times B_{\frac{1}{64}}\left(\frac{2x_e}{t_e}\right)
\end{equation}
 
\medskip

\noindent{\bf Step 3: Spatial transport by an intermediate velocity.} 

Now, we allow the lower bounds \eqref{e.s2_explicit} to evolve by transport; that is we sustain the lower bound over a time interval $[s_2, \sigma_1]$, where $s_2$ is the arrival time of the previous diffusion stage and $\sigma_1 = s_2 + \sfrac{t_e}{2}$. More precisely, the corresponding trajectory is simply
\begin{equation*}
    v_s = \frac{2x_e}{t_e} \qquad \text{and} \qquad x_s = \int_{s_2}^{s} v_t \dd t = \left(\frac{2x_e}{t_e}\right)(s-s_2) \qquad \text{for each } s_2 \le s \le \sigma_1.
\end{equation*}
We aim to apply \Cref{l.transport} to this trajectory using $r = \frac{1}{64}$ and $\delta$ coming from \eqref{e.s2_explicit}. Once again, the smallness condition is provided by the choice of $t_{\rm core}$ and we conclude
\begin{equation}
    f(\sigma_1,x,v),\,h(\sigma_1,x,v) \gtrsim 1 \qquad \text{for all }(x,v)\in B_\frac{1}{256}(x_e)\times B_{\frac{1}{256}}\left(\frac{2x_e}{t_e}\right),
\end{equation}
since the transport trajectory was designed to arrive at $x_e$ precisely at $\sigma_1$. From the previous stage, we may vary the arrival time of the diffusion as $s_2 \in [\sfrac{t_e}{4},\sfrac{t_e}{2}]$ to deduce the more explicit bound
\begin{equation}\label{e.sigma1_explicit}
    f(t,x,v), h(t,x,v) \gtrsim 1 \qquad \text{for all }(t,x,v)\in \left[\frac{3t_e}{4},t_e\right] \times B_\frac{1}{256}(x_e)\times B_{\frac{1}{256}}\left(\frac{2x_e}{t_e}\right).
\end{equation}
The purpose of staggering the arrival time $\sigma_1$ is a corresponding ellipticity core at the final position $x_e$. As in Step 1, the nonlocality of $A$ then spreads \eqref{e.sigma1_explicit} to an ellipticity bound globally in $v$ for all times near $t_e$:
\begin{equation}\label{e.A[fh]_lower2}
    \cA(t,x,v) \ge A[fh](t,x,v) \gtrsim \frac{1}{\vv^3} \qquad \text{for all }(t,x,v) \in \left[\frac{3t_e}{4},t_e\right] \times B_\frac{1}{256}(x_e)\times \R^3.
\end{equation}

\medskip
\noindent{\bf Step 4: Spreading to the final velocity by diffusion.}

Finally, we spread the lower bound \eqref{e.sigma1_explicit} from the intermediate velocity $\sfrac{2x_e}{t_e}$ to the terminal velocity $v_e$ using diffusion. The argument is identical to Step 2 where the arbitrary diffusion times $s_1$ and $s_2$ have been replaced by $\sigma_1$ and $\sigma_2$. For completeness, we give the explicit details, although the savvy reader may skip this step.

For the moment, let $[\sigma_1,\sigma_2] \subset \left[\sfrac{3t_e}{4},t_e\right]$ be arbitrary; we will choose $\sigma_2 = t_e$ below. The corresponding diffusion trajectory is then $(x_\sigma,v_\sigma)$ given by
\begin{equation}
    x_\sigma = x_e + \int_{\sigma_1}^\sigma v_t \dd t \qquad \text{and} \qquad v_\sigma = \frac{2x_e}{t_e} + \left(v_e - \frac{2x_e}{t_e}\right)\left(\frac{\sigma - \sigma_1}{\sigma_2 -\sigma_1}\right).
\end{equation}
After a time translation by $\sigma_1$, we aim to apply \Cref{l.spreading} to this trajectory with $\tau = \sigma_2 - \sigma_1$, $r = \frac{1}{512}$, and $\delta$ defined by \eqref{e.sigma1_explicit}. 

Once more, the initial lower bound is provided the previous stage \eqref{e.sigma1_explicit} through the assumption $[\sigma_1,\sigma_2] \subset \left[\sfrac{3t_e}{4},t_e\right]$ and the necessary smallness is provided by the choice of $t_{\rm core}$. To apply \Cref{l.spreading}, we verify that $\norm{\lambda}_{\rm path} < \infty$ which, by \eqref{e.A[fh]_lower2}, reduces to showing that $B_{\frac{1}{512}}(x_\sigma) \subset B_{\frac{1}{256}}(x_e)$ for each $\sigma_1 \le \sigma \le \sigma_2$. We claim that this holds provided $\sigma_1$ and $\sigma_2$ are further required to satisfy the condition
\begin{equation}\label{e.small_time_interval_sigma}
     \sigma_2 - \sigma_1 \leq \frac{1}{4096}\min\left\{\frac{1}{\brak{v_e}}, \frac{t_e}{2\brak{x_e}}\right\}.
 \end{equation}
 Indeed, with this assumption, we compute
 \begin{equation}
     \abs{x_\sigma - x_e} \le \int_{\sigma_1}^{\sigma_2} \abs{v_t} \dd t \le (\sigma_2 -\sigma_1)\max\left\{\abs{v_e}, \frac{2\abs{x_e}}{t_e}\right\} \le \frac{1}{4096}.
 \end{equation}
Consequently, it follows that $B_{\frac{1}{512}}(x_\sigma) \subset B_{\frac{1}{256}}(x_e)$ for each $\sigma_1 \le \sigma \le \sigma_2$. Applying \Cref{l.spreading}, choosing $\sigma_2 = t_e$, and choosing $\sigma_1 > \sfrac{3t_e}{4}$ to saturate \eqref{e.small_time_interval_sigma}, we conclude
\begin{equation}
    f(t_e,x,v),\, h(t_e,x,v) \gtrsim 1 \qquad\text{for all }(x,v) \in B_{\frac{1}{2048}}(x_{t_e})\times B_{\frac{1}{2048}}(v_e).
\end{equation}
One final admissibility check guarantees $x_e \in B_{\frac{1}{4096}}(x_e) \subset B_{\frac{1}{2048}}(x_{t_e})$ and we conclude
\begin{equation}
    f(t_e,x_e,v_e),\, h(t_e,x_e,v_e) \gtrsim 1.
\end{equation}
Since the point $(t_e,x_e,v_e)$ was arbitrary and the implicit constants clearly depend locally uniformly on $t_e$, $x_e$, and $v_e$, the proof is complete.
\end{proof}

\section{Global Existence of Weak Solutions}\label{sec:existence}

In this section, we begin the proof of \Cref{thm:global_existence} by constructing global weak solutions to~\eqref{e.LFD} that satisfy both of our {\em a priori} estimates, \Cref{thm:upper_bounds} and \Cref{t.nonvacuum}, and are compatible with the existing conditional regularity theory for kinetic equations. This will enable us to deduce that the constructed weak solutions are classical for positive times. We begin by defining our notion of weak solution precisely:
\begin{definition}\label{def:weak-solutions}
    We say that $f$ is a weak solution to~\eqref{e.LFD} on $(0,T)\times\T^3\times\R^3$ for $0 < T \le \infty$ if, for some $m>2$,
    \begin{equation*}
        0 \leq f \leq 1 \quad \text{a.e.}, \qquad f \in L^\infty_{\rm loc}([0,T];L^\infty_m), \qquad \nabla_v f \in L^2_{\mathrm{loc}}\big((0,T)\times\T^3\times\R^3\big),
    \end{equation*}
    and, for every
    $\psi\in C_c^\infty((0,T)\times\T^3\times\R^3)$,
    \begin{equation}
        \iiint f(\partial_t+v\cdot\nabla_x)\psi \dd v\dd x\dd t = \iiint \nabla_v\psi\cdot\left(A[f(1-f)]\nabla_v f - \nabla a[f]f(1-f)\right)\dd v\dd x\dd t.
    \end{equation}
\end{definition}

By \Cref{lem:coefficient_bounds}, the assumptions on $f$ ensure that the nonlocal coefficients $A[f(1-f)]$ and $\nabla a[f](1-f)$ are globally bounded and the diffusion coefficient $A[f(1-f)]$ is nonnegative. Thus, the above definition guarantees that $f$ solves a \emph{linear} Kolmogorov-type equation with bounded, measurable coefficients in a relatively standard weak sense (see \Cref{defn:linear_weak_solution} below). Consequently, once local coercivity of $A[f(1-f)]$ is known, the tools developed for linear equations---collected in \Cref{appendix:linear}---apply to these weak solutions of LFD.

The main result of this section is the following proposition:
\begin{proposition}\label{prop:existence}
    Fix $m > 2$ and let  $f_{\rm in}\in L^\infty_m(\T^3\times \R^3)$ satisfy $0 \le f_{\rm in} \le 1$ and the non-vacuum condition from \Cref{d.nonvacuum}. Then, there exists a global weak solution $f\colon \R^+\times \T^3\times \R^3 \to [0,1]$ of \eqref{e.LFD} in the sense of \Cref{def:weak-solutions}. Moreover:
    \begin{enumerate}
        \item \textbf{A priori bounds}: For each $T > 0$, $f\in L^\infty(0,T;L^\infty_m)$ and $f$ satisfies the estimates contained in \Cref{thm:upper_bounds} and \Cref{t.nonvacuum}.
        
        \item \textbf{Behavior at initial time}: For each $T > 0$, $f \in C([0,T];L^2_{x,v})$ with
        \begin{equation*}
        \begin{aligned}
            (\partial_t + v\cdot\nabla_x)f \in L^2([0,T]\times&\T^3;H^{-1}_v), \quad \int_0^T\iint \nabla_v f\cdot A[f(1-f)]\nabla_v f\dd v\dd x \dd t <\infty,\\[5pt]
                &\text{and} \quad \lim_{t \to 0^+}\norm{f(t) - f_{\rm in}}_{L^2_{x,v}} = 0.
        \end{aligned}
        \end{equation*}
    \end{enumerate}
\end{proposition}

Although global distributional solutions for LFD were constructed by Sampaio \cite{Sampaio1}, those may fail to satisfy the qualitative decay and regularity assumptions inherent in \Cref{def:weak-solutions}. For example, the characteristic function of a ball in velocity is a stationary distributional solution to LFD in Sampaio's sense, and cannot satisfy smoothing estimates. The qualitative regularity condition $\nabla_v f \in L^2_{\rm loc}$ is fundamental for standard smoothing estimates; it permits the use of $f$ as a test function in its weak formulation, essentially giving access to regularity theory based on energy methods.

Despite the rather strong {\em a priori} estimates shown for classical solutions to the LFD equation, it is not so straightforward to construct solutions. The main difficulty is designing an appropriate linearization scheme that maintains the crucial Pauli exclusion principle. This difficulty has been addressed by \cite{GoldingGualdaniZamponi} in the spatially homogeneous setting, and we adapt this approach to the inhomogeneous setting here.

The proof of \Cref{prop:existence} consists of five steps. In Step 1, we show existence and uniqueness for a semilinear viscous problem using the Banach fixed-point theorem. Step 2 establishes a decay estimate using the maximum principle. Step 3 provides local existence for the full nonlinear problem with an artificial viscosity term using the Schauder fixed-point theorem. Step 4 argues that the constructed solutions are smooth. Finally, in Step 5, the \emph{a priori} estimates provide global continuation and compactness necessary to pass to the limit and remove the artificial viscosity.

\subsection*{Step 1: Existence for a Semilinear Parabolic Problem}

For convenience, we introduce a semilinearized version of $\QLFD$:
\begin{equation*}
    \QLFD(g,f)\coloneqq \nabla_v\cdot \left(A[\tilde g]\nabla_v f - \nabla a[g]\tilde f\right), \qquad \tilde f = f(1-f), \qquad \tilde g = g(1-g).
\end{equation*}
The collision operator appearing in \eqref{e.LFD} is exactly $\QLFD(f,f)$. The advantage is that the operator $\QLFD(g,f)$ clearly isolates the nonlocal behavior of $\QLFD$; for fixed $g$, $\QLFD(g,\cdot)$ is a local elliptic operator. We emphasize that although $\QLFD(g,\cdot)$ is local, it remains semilinear with possibly degenerate ellipticity. To combat the degenerate ellipticity, we add an artificial viscosity term $\kappa \Delta_v f$ and begin by constructing solutions to the following \emph{semilinear} Cauchy problem:
\begin{equation}\label{eq:semilinear}
\begin{aligned}
    (\partial_t + v \cdot \nabla_x) f &= \QLFD(g,f) + \kappa \Delta_v f \qquad &&\text{in }(0,T)\times \T^3 \times \R^3\\
    f(0,x,v) &= f_{\rm in}(x,v)  &&\text{on } \T^3 \times \R^3.
\end{aligned}
\end{equation}
Here $g \in L^\infty(0,T;L^\infty_m)$ is fixed, $f_{\rm in} \in \Lrapid$ (defined in ~\eqref{e.Lrapid}), and both satisfy Pauli exclusion. 
More precisely, in this step we show:
\begin{lemma}\label{lem:semilinear}
    Fix any $T>0$, $\kappa>0$, and $m > 2$. Suppose that $f_{\rm in} \in \Lrapid(\T^3\times \R^3)$, $g\in L^\infty(0,T;L^\infty_m)$, and $0 \le g, \,f_{\rm in} \le 1$. Then, there exists a unique weak solution $f\colon[0,T]\times \T^3 \times \R^3 \to [0,1]$ to \eqref{eq:semilinear} satisfying $f \in C([0,T];L^2_{x,v}) \cap L^2_{t,x}H^1_v$.
\end{lemma}

\begin{proof}
    The proof of existence and uniqueness follows from a suitable application of the Banach fixed-point theorem and linear estimates. Let us introduce an auxiliary linear problem. For any fixed $h \in L^2_{t,x,v}$, define the coefficients
    \begin{equation}
        \cA \coloneqq A[\tilde g] + \kappa\Id \qquad \text{and} \qquad \cB_h \coloneqq \tilde h_+ \nabla a[g].
    \end{equation}
    Here $\tilde h_+ = \max(0,h(1-h))$ is a nonnegative, bounded function for any measurable $h$. For any given $\kappa$, $g$, and $h$, we study the linear equation,
    \begin{equation}\label{eq:linear_problem_with_force}
        (\partial_t + v\cdot \nabla_x) f = \nabla_v \cdot \left(\cA\nabla_v f\right) - \nabla_v \cdot \cB_h \qquad \text{with} \qquad f(0) = f_{\rm in}.
    \end{equation}
    The assumptions on $g$ imply that $\cA \in L^\infty([0,T]\times \T^3 \times \R^3)$ and $\cA \ge \kappa \Id$ by \Cref{lem:coefficient_bounds}. Additionally, $\cB_h \in L^2([0,T]\times \T^3\times \R^3)$ whenever $h \in L^2([0,T]\times \T^3\times \R^3)$. \Cref{lem:linear_existence} therefore yields a unique solution to \eqref{eq:linear_problem_with_force} for any $h \in L^2([0,T]\times\T^3\times\R^3)$.

    Thus, we define $\Phi$ as the (nonlinear) solution operator associated to~\eqref{eq:linear_problem_with_force}; that is, $\Phi(h) = f$. We seek a solution to the nonlinear problem~\eqref{eq:semilinear} as a fixed point of $\Phi$. To this end, let $X_T$ be the Banach space $C([0,T];L^2(\T^3\times \R^3))$. Note that $\Phi:X_T \to X_T$ by the linear estimates for \eqref{eq:linear_problem_with_force} (see \Cref{lem:linear_existence} and \Cref{lem:linear_continuity}). Next, we verify that $\Phi$ is a contraction.
    Suppose $h_1,\,h_2\in X_T$ and let $f_1$ and $f_2$ denote the corresponding weak solutions to \eqref{eq:linear_problem_with_force}. Setting $w = f_1 -f_2$, the standard linear $L^2_{x,v}$ estimate for $w$ yields
    \begin{equation*}
        \frac{1}{2}\frac{\dd }{\dd t} \norm{w}_{L^2}^2 + \frac{\kappa}{2} \norm{\nabla_v w}^2_{L^2} \le \frac{1}{2\kappa}\norm{\cB_{h_1} -\cB_{h_2}}_{L^2}^2.
    \end{equation*}
    The key point is the function $s \to \max(0,s(1-s))$ is Lipschitz (with constant $1$). In particular, 
    \begin{equation*}
        \norm{\cB_{h_1} -\cB_{h_2}}_{L^2}^2 \le \norm{\nabla a[g]}_{L^\infty}^2 \norm{h_1 - h_2}_{L^2}^2 \le C \norm{h_1 - h_2}_{L^2}^2
    \end{equation*}
    Thus, integrating in time yields a contraction estimate:
    \begin{equation*}
        \norm{\Phi(h_1) -\Phi(h_2)}_{X_\tau}^2 = \sup_{0 < t < \tau} \norm{w}_{L^2}^2 \le C\int_0^\tau \norm{h_1(t) -h_2(t)}_{L^2}^2 \dd t \le C \tau\norm{h_1 - h_2}_{X_\tau}^2,
    \end{equation*}
    where $C$ is a constant depending only on $g$ and $\kappa$. Consequently, for $\tau$ sufficiently small depending only on $g$ and $\kappa$, the Banach fixed-point theorem yields a unique fixed point of $\Phi$ in $X_\tau$. Since the choice of $\tau$ is independent of the initial datum, iterating the argument yields a unique fixed point in $X_T$.
    By definition, the fixed point $f$ belongs to $C([0,T];L^2_{x,v}) \cap L^2_{t,x}H^1_v$ and satisfies the following semilinear equation:
    \begin{equation}\label{eq:semi_linear_auxiliary}
    (\partial_t + v\cdot \nabla_x) f = \kappa\Delta_{v} f + \nabla_v \cdot \left(A[\tilde g]\nabla_v f - \nabla a[g]\tilde f_+ \right) \qquad \text{with} \qquad f(0) = f_{\rm in}.
    \end{equation}
    It remains to remove the positive part on $\tilde f$ and deduce that $f$ satisfies the crucial Pauli exclusion bound $0 \le f \le 1$. Due to the low level of regularity, we take the Stampacchia approach to the maximum principle.
    
    Note that $f_- = \min(0,f)$ is a valid test function for \eqref{eq:semi_linear_auxiliary} since $\nabla_v f_- = \mathbbm{1}_{\set{f < 0}}\nabla_v f \in L^2$. Additionally, $f_-$ and $\tilde f_+$ have disjoint supports. Therefore, a standard truncation argument justifies
    \begin{equation*}
        \frac{1}{2}\frac{\dd}{\dd t} \norm{f_-(t)}_{L^2_{x,v}}^2 = -\iint \cA\nabla_v f_- \cdot \nabla_v f_- \dd v \dd x \le 0.
    \end{equation*}
    Thus, $t \to \norm{f_-(t)}_{L^2}^2$ is decreasing, nonnegative, and initially zero by the assumption that $f_{\rm in} \ge 0$. We conclude $f_- = 0$, or, equivalently, $f \geq 0$.

    Similarly, we note $(f-1)_+ = \max(f-1,0)$ is a valid test function for \eqref{eq:semi_linear_auxiliary} because $\nabla_v (f-1)_+ = \mathbbm{1}_{\set{f > 1}}\nabla_v f$. Again, $(f-1)_+$ and $\tilde f_+$ have disjoint supports and another standard truncation argument justifies
    \begin{equation*}
        \frac{1}{2}\frac{\dd}{\dd t} \norm{(f-1)_+}_{L^2_{x,v}}^2 =  - \iint \cA\nabla_v (f-1)_+ \cdot \nabla_v (f-1)_+ \dd v \dd x \le 0.
    \end{equation*}
    By the assumption that $f_{\rm in} \le 1$, we conclude $(f-1)_+ = 0$ and, thus, $f \leq 1$. This verifies Pauli exclusion for $f$, and implies $\tilde f_+ = \tilde f$, so that $f$ solves the desired equation.
\end{proof}

\subsection*{Step 2: A Maximum Principle Estimate}

To remove the $g$ in \eqref{eq:semilinear}, we use a second fixed point argument. We require, however, a more substantial estimate of the collision operator that will persist even for \eqref{e.LFD}.
\begin{lemma}\label{lem:maximum_principle}
    Fix $T > 0$, $\kappa > 0$, and $m > 2$. Suppose $f_{\rm in} \in \Lrapid(\T^3\times \R^3)$ and $g\in L^\infty(0,T;L^\infty_m)$ satisfy $0 \le f_{\rm in},\,g\le 1$.
    Then, for every $\ell > 0$, the unique solution $f$ to \eqref{eq:semilinear} satisfies
    \begin{equation}
         \norm{f(t)}_{L^\infty_{\ell}} \le \exp\bigl\{C(\kappa + \Lambda)t\bigr\}
         \norm{f_{\rm in}}_{L^\infty_\ell}, 
    \end{equation}
    where $C$ depends only on $m$ and $\ell$ and where $\displaystyle \Lambda \coloneqq \norm{g(t)}_{L^\infty(0,T;L^\infty_m)}$. In particular, $f(t) \in \Lrapid$ for each $t \in [0,T]$.
\end{lemma}

\begin{proof}
    The proof is by the maximum principle for (degenerate) parabolic equations with bounded measurable coefficients. Since we work with weak solutions, we will take a Stampacchia approach; namely, for $b(t,v)$ an appropriately chosen smooth barrier, we use $(f - b)_+$ as a test function in the weak formulation for $f$ and show that $\norm{(f-b)_+}_{L^2}^2$ satisfies a differential inequality forcing it to vanish. We begin with computations that are independent of any fine structure of $b$. We  assume only that $b = b(t,v)$ is nonnegative, smooth, and decays as $\abs{v}\to \infty$. Since the barrier is independent of $x$, the transport term vanishes and we find
    \begin{equation}\label{eq:long_bound_split}
    \begin{split}
        \frac{1}{2}\frac{\dd}{\dd t}\iint (f- b)_+^2 \dd v \dd x &= \iint (f - b)_+(\partial_t f - \partial_t b) \dd v\dd x\\
            &= -\underbrace{\iint (A[\tilde g]+\kappa \Id)\nabla_v f \cdot \nabla_v (f-b)_+ \dd v \dd x}_{\mathrm{Term\; I}}\\
            &\qquad + \underbrace{\iint \nabla_v(f-b)_+ \cdot \nabla a[g]\tilde f \dd v \dd x}_{\mathrm{Term\; II}} - \underbrace{\iint \partial_t b\ (f-b)_+    \dd v \dd x}_{\mathrm{Term\;III}}.
    \end{split}
    \end{equation}
    A simple but tedious algebraic computation yields the following identity:
    \begin{equation}\label{eq:long_bound}
    \begin{split}
        \frac{1}{2}\frac{\dd}{\dd t}\iint (f- b)_+^2 \dd v \dd x &= -\underbrace{\iint \left(A[\tilde g]+\kappa \Id\right)\nabla_v (f-b)_+ \cdot \nabla_v(f-b)_+ \dd v \dd x}_{\rm Term\;I}\\
            &\qquad+ \underbrace{\iint \left[(A[\tilde g]+\kappa \Id):(\nabla_v^2 b) + \nabla a[\tilde g]\cdot \nabla_v b\right]\ (f-b)_+ \dd v \dd x}_{\mathrm{Term\;I}}\\
            &\qquad+ \underbrace{\iint (f-b)_+\left[g(b-b^2) - (1-2b)\nabla a[g]\cdot\nabla_v b\right] \dd v \dd x}_{\rm Term \; II}\\
            &\qquad+ \underbrace{\frac{1}{2}\iint (f-b)_+^2\left[g(1-2b) +2\nabla a[g]\cdot\nabla_v b\right] \dd v \dd x}_{\rm Term \; II}\\
            &\qquad - \underbrace{\frac{1}{3}\iint (f-b)_+^3g \dd v \dd x}_{\rm Term \; II} - \underbrace{\iint \partial_t b\ (f-b)_+    \dd v \dd x}_{\mathrm{Term\;III}}.
    \end{split}
    \end{equation}
    The above computation is obtained by repeatedly adding and subtracting $b$, integrating by parts, and using the identities $-\Delta a[g] = g$ and $\nabla \cdot A[\tilde g] = \nabla a[\tilde g]$. The repeated labels indicate which term in \eqref{eq:long_bound_split} produces the corresponding contribution in \eqref{eq:long_bound}.
    Next, we note that both $f$ and $g$ satisfy Pauli exclusion so that $0\le g \le 1$ and $0 \le (f - b)_+ \le f \le 1$. Moreover, \Cref{lem:coefficient_bounds} yields
    \begin{equation*}
        \norm{A[\tilde g]}_{L^\infty} + \norm{\nabla a[g]}_{L^\infty} + \norm{\nabla a [\tilde g]}_{L^\infty}
        \le C\norm{g(t)}_{L^\infty_m} \le C\Lambda,
    \end{equation*}
    where $C$ is a constant depending only on $m$. Thus, discarding the advantageous diffusion term in \eqref{eq:long_bound}, there holds
    \begin{equation}\label{eq:shorter_bound}
    \begin{aligned}
        \frac{\dd}{\dd t}\iint &(f- b)_+^2 \dd v \dd x
        \le \iint (f-b)_+^2 \dd v \dd x\\   
            &+ C\left(\kappa + \Lambda\right)\iint \left(\abs{\nabla_v^2 b} + \abs{\nabla_v b} + b \right) \ (f-b)_+ \dd v \dd x - 2\iint \partial_t b\ (f-b)_+ \dd v\dd x.
    \end{aligned}
    \end{equation}
    Finally, we specify a simple functional form for the barrier $b$ with a parameter $\alpha \in \R_+$ to be chosen:
    \begin{equation}\label{eq:barrier_definition}
        b(t,v) \coloneqq \|f_{\rm in}\|_{L^\infty_\ell} e^{\alpha t}\brak{v}^{-\ell}.
    \end{equation}
    Computing derivatives of $b$ yields
    \begin{equation*}
        \partial_t b = \alpha b, \qquad \nabla_v b = -\ell \left(\frac{v}{\brak{v}^2}\right)b, \qquad \text{and} \qquad \nabla^2_v b = \left[\ell(\ell + 2)\left(\frac{v}{\brak{v}^2} \tens \frac{v}{\brak{v}^2}\right) - \frac{\ell\Id}{\brak{v}^2}\right]\,b(t,v).
    \end{equation*}
    Thus, substituting the form of the barrier into \eqref{eq:shorter_bound}, we find
    \begin{equation*}
        \frac{\dd}{\dd t}\iint (f- b)_+^2 \dd v \dd x \le \iint (f-b)_+^2 \dd v \dd x + \left[ \widetilde C(\kappa + \Lambda) - \alpha\right]\iint b  \ (f-b)_+ \dd v \dd x,
    \end{equation*}
    where $\widetilde C$ is a constant depending only on $\ell$ and $m$. Choosing $\alpha = \widetilde C(\kappa + \Lambda)$, and applying Gr\"onwall's inequality to the map $t \to \norm{(f(t)-b(t))_+}_{L^2}^2$, we conclude $\norm{(f-b)_+}_{L^2}^2 = 0$ for all time. Unraveling definitions, this implies the pointwise bound $f(t,x,v) \le b(t,v)$, which combined with \eqref{eq:barrier_definition} is the desired estimate.
\end{proof}

\subsection*{Step 3: Existence and Continuation for a Fully Nonlinear Approximation}

We are now ready to prove the short time existence for a fully nonlinear approximation. Fix $m > 2$, $\kappa > 0$, and rapidly decaying initial data $f_{\rm in} \in \Lrapid$ satisfying $0 \le f_{\rm in} \le 1$. By \Cref{lem:semilinear}, for any $T>0$ and $g \in L^\infty(0,T;L^\infty_m)$, there is a unique solution to the semilinear problem
\begin{equation}\label{eq:semilinear2}
\begin{aligned}
    (\partial_t + v\cdot \nabla_x) f &= \kappa \Delta_v f + \QLFD(g,f)\\
        f(0) &= f_{\rm in}
\end{aligned}
\end{equation}
For $\Lambda,\,T > 0$, fix the Banach space $Y_T \coloneqq L^2([0,T]\times \T^3\times \R^3)$ and the convex set
\begin{equation}\label{eq:definition_K}
    K_{T,\Lambda} \coloneqq \set{g \in Y_T \ \mid \ 0\le g \le 1,\;\norm{g}_{L^\infty(0,T;L^\infty_m)} \le \Lambda}.
\end{equation}
Consequently, the (nonlinear) solution operator $\Psi$ to \eqref{eq:semilinear2} is well-defined for $g \in K_{T,\Lambda}$ for any $T, \Lambda > 0$ via the formula $\Psi(g) = f$. To obtain existence of a fixed point, we make the following four observations about the mapping $\Psi$ and the set $K_{T,\Lambda}$:
    \begin{itemize}
    
        \item The set $K_{T,\Lambda}$ is a convex, closed, and nonempty subset of $Y_T$ for each $T,\,\Lambda > 0$.
        
        \item The inclusion $\Psi(K_{T,\Lambda}) \subset Y_T$ is compact for any $\Lambda,\,T>0$ (see \Cref{lem:linear_compactness}).
        
        \item The nonlinear map $\Psi:K_{T,\Lambda} \to Y_T$ is continuous. Indeed, take a sequence $g_k \to g$ in $Y_T$ and let $f_k = \Psi(g_k)$ denote the corresponding images. Since $\Psi(K_{T,\Lambda})$ is precompact, it suffices to show that $\Psi(g)$ is the unique limit point of $f_k$. Assume that $f$ is a limit point of $f_k$ in $Y_T$ and after passing to a subsequence we may assume $f_k \to f$ and $g_k \to g$ pointwise as well. Using the uniform bound $\norm{g_k}_{L^\infty(0,T;L^\infty_m)}\le \Lambda$ and the Lebesgue dominated convergence theorem yields
        \begin{equation*}
            A[\tilde g_k] \to A[\tilde g] \qquad \text{and} \qquad \nabla a[g_k](1-f_k)f_k \to \nabla a[g](1-f)f \qquad \text{pointwise and strongly in $L^2_{\rm loc}$.} 
        \end{equation*}
        Consequently, the stability of weak solutions provided by \Cref{lem:linear_stability} implies $f$ is a solution to \eqref{eq:semilinear2} with coefficients determined by $g$. By uniqueness for the semilinear problem, we conclude $f = \Psi(g)$. Since $f$ was an arbitrary limit point, we conclude $\Psi$ is continuous.
        
        \item Lastly, for $T$ and $\Lambda$ chosen appropriately, $\Psi(K_{T,\Lambda}) \subset K_{T,\Lambda}$. Indeed, choosing
        \begin{equation}
            \Lambda = 2 \norm{f_{\rm in}}_{L^\infty_m} \qquad \text{and} \qquad T = \frac{\log(2)}{C(\Lambda + \kappa)},
        \end{equation}
        where $C=C(m)$ is the constant from applying \Cref{lem:maximum_principle} with $m = \ell$, there holds
        \begin{equation}\label{eq:max_principle_estimate}
            \sup_{0 < t < T}\norm{f}_{L^\infty_m} \le \exp \left\{C(\Lambda + \kappa) T \right\}\norm{f_{\rm in}}_{L^\infty_m} \le  2\norm{f_{\rm in}}_{L^\infty_m}
            = \Lambda.
        \end{equation}
    \end{itemize}
    These four points enable an application of Schauder's fixed point theorem to conclude that $\Psi$ has a fixed point $f \in K_{T,\Lambda}$; in other words, there is an $f\in K_{T,\Lambda}$ which is a weak solution to the fully nonlinear problem with an added viscous term:
    \begin{equation}\label{eq:fully_nonlinear_viscous}
    \begin{aligned}
    (\partial_t + v\cdot \nabla_x) f &= \kappa\Delta_{v} f + \QLFD(f)\\
        f(0) &= f_{\rm in}.
    \end{aligned}
    \end{equation}
    Note that by \Cref{lem:maximum_principle}, the constructed solution decays rapidly for all $t\in[0,T]$.
    Additionally, the existence time of $f$ is explicitly $T \sim (\norm{f_{\rm in}}_{L^\infty_m} + \kappa)^{-1}$. If we can apply \Cref{thm:upper_bounds} to the constructed solution, the $\norm{f(t)}_{L^\infty_m}$ cannot blow up and we obtain global solutions to \eqref{eq:fully_nonlinear_viscous} for each $\kappa > 0$.

\subsection*{Step 4: Qualitative Regularity of the Viscous Weak Solutions}

    We note that, for each fixed $\kappa > 0$, the weak solution $f$ of the viscous problem \eqref{eq:fully_nonlinear_viscous} is sufficiently regular to justify the estimates below. In particular, for rapidly decaying data $f_{\rm in}\in \Lrapid$, \Cref{lem:maximum_principle} implies $f\in L^\infty(0,T;L^\infty_m)$ for each $m > 0$ and consequently $f\in C^\infty([\tau,T]\times \T^3;\cS(\R^3))$ for each $0<\tau < T$, i.e. $f$ is immediately smooth and Schwartz class. This follows by the same regularity bootstrap used in \Cref{sec:regularity}. In fact, the argument is strictly easier here: the additional term $\kappa\Delta_v f$ supplies non-degenerate diffusion at large velocities that eliminates the need for the complicated change of variables in \Cref{sec:regularity}.

    The regularity constants may depend on $\kappa$ as well as $\tau$ and $T$. This is harmless, since this step is used only to justify that the estimates from \Cref{thm:upper_bounds} and \Cref{t.nonvacuum} apply to the smooth approximate solutions for positive times.

    \subsection*{Step 5: Existence of Global Weak Solutions to LFD}

    We are now prepared to construct our global weak solutions in the sense of \Cref{def:weak-solutions} and complete the proof of \Cref{prop:existence}.

    \begin{proof}[Proof of \Cref{prop:existence}]
    Fix any $\kappa \in (0,1)$ and $R > 0$ and pick a cutoff function $\chi_R \in C^\infty_c(\R^3)$ satisfying
    \begin{equation*}
        0 \le \chi_R \le 1, \qquad \chi_R \equiv 1 \text{ on }B_R, \qquad \text{and} \qquad \chi_R \equiv 0 \text{ on }\R^3\setminus B_{2R}.
    \end{equation*}
    Setting $f^R_{\rm in} = \chi_R(v)f_{\rm in}(x,v)$, we note 
    \begin{equation*}
        f^R_{\rm in} \in \Lrapid, \qquad \norm{f_{\rm in}^R}_{L^\infty_m} \le \norm{f_{\rm in}}_{L^\infty_m}, \qquad \text{and} \qquad f^R_{\rm in} \to f_{\rm in} \quad \text{strongly in }L^2_{x,v}\text{ and pointwise}.
    \end{equation*}
    Iteratively applying Step 3 and the smoothing described in Step 4, we obtain a Schwartz class function $f^{\kappa,R} \in C^\infty((0,T_{\kappa,R})\times\T^3;\cS(\R^3))$ defined on a maximal interval of existence $[0,T_{\kappa,R})$ satisfying
    \begin{equation}\label{eq:fully_nonlinear_approximation}
    \begin{aligned}
    (\partial_t + v\cdot \nabla_x) f^{\kappa,R} &= \kappa\Delta_{v} f^{\kappa,R} + \QLFD(f^{\kappa,R}) &&\text{in} \quad(0,T_{\kappa,R})\times \T^3\times \R^3\\
        f^{\kappa,R}(0,x,v) &= f^R_{\rm in}(x,v) &&\text{on}\quad\T^3\times \R^3.
    \end{aligned}
    \end{equation}
    Moreover, $0 \le f^{\kappa,R}\le 1$ and the maximal time of existence as a Schwartz class solution, if it is finite, is characterized by
    \begin{equation}\label{eq:maximal_existence}
        \lim_{t \to T_{\kappa,R}^{-}} \norm{f^{\kappa,R}(t)}_{L^\infty_m} = + \infty.
    \end{equation}
    Let us first establish that $T_{\kappa,R}  = \infty$. Applying our main {\em a priori} estimate, \Cref{thm:upper_bounds}, or more properly \Cref{rem:apriori_upper_bounds} on $[s,T_{\kappa,R})$ and passing to the limit $s\to0^+$, we deduce that 
    \begin{equation}\label{eq:uniform_bound}
        f^{\kappa,R}(t,x,v) \le \frac{\norm{f^R_{\rm in}}_{L^\infty_m}\exp\bigl\{(K + C\kappa)t\bigr\}}{\abs{v}^m}, \quad  \text{for any} \quad 0  \leq t < T_{\kappa,R}.
    \end{equation}
    Clearly, this is incompatible with \eqref{eq:maximal_existence}.  We conclude $T_{\kappa,R} = \infty$.

    We next show the coercivity needed for strong compactness. We appeal to our second main {\em a priori} estimate, \Cref{t.nonvacuum}, or more properly to \Cref{p.nonvacuum}, which holds when a viscous term is added. Suppose that $f_{\rm in}$ is non-vacuum at $(x_0,v_0)$ with radius $r$ and size $\delta$. For $R> 2 (|v_0| + r)$, $f^R_{\rm in} = f_{\rm in}$ on the phase space ball in \Cref{d.nonvacuum}. Consequently, we conclude that 
    \begin{equation}\label{eq:uniform_coercivity}
        A[f^{\kappa,R}(1-f^{\kappa,R})] \ge \lambda(t,x,v)\Id,    
    \end{equation}
    where $\lambda$ is uniformly positive on compact sets of $(0,\infty)\times \T^3\times \R^3$, independent of $R$ and $\kappa$. 
    
    Our final uniform estimate follows from the standard $L^2_{x,v}$ energy estimate. Multiplying $\eqref{eq:fully_nonlinear_approximation}$ by $f^{\kappa,R}$, integrating by parts, and using $-\Delta a[f^{\kappa,R}] = f^{\kappa,R}$, we find that for any $0 < T$,
    \begin{equation}\label{eq:uniform_dissipation}
    \begin{aligned}
        \norm{f^{\kappa,R}(T)}_{L^2}^2 + 2\int_0^T\iint A[f^{\kappa,R}(1-f^{\kappa,R})]\nabla_v &f^{\kappa,R}\cdot\nabla_v f^{\kappa,R} \dd v \dd x \dd t\\
            &\le \norm{f^R_{\rm in}}_{L^2}^2 + \int_0^T\norm{f^{\kappa,R}(t)}_{L^2}^2 \dd t.
    \end{aligned}
    \end{equation}

    We now pass to the limit as $\kappa \to 0^+$ and $R \to \infty$ using the uniform bounds \eqref{eq:uniform_bound}, \eqref{eq:uniform_coercivity}, and \eqref{eq:uniform_dissipation}. Introducing the frozen coefficients
    \begin{equation}\label{eq:frozen_coefficients}
        \cA^{\kappa,R}\coloneqq \kappa\Id +A[f^{\kappa,R}(1-f^{\kappa,R})] \quad\text{and} \quad \cB^{\kappa,R} \coloneqq -f^{\kappa,R}(1-f^{\kappa,R})\nabla a[f^{\kappa,R}],
    \end{equation}
    the nonlinear problem \eqref{eq:fully_nonlinear_approximation} is equivalent to the linear problem
    \begin{equation}\label{eq:frozen_coefficients_linear}
        (\partial_t + v\cdot \nabla_x) f^{\kappa,R} = \nabla_v \cdot \left(\cA^{\kappa,R}\nabla_v f^{\kappa,R}\right) + \nabla_v \cdot \left(\cB^{\kappa,R}\right).
    \end{equation}
    The uniform estimate \eqref{eq:uniform_bound} implies $\norm{\cA^{\kappa,R}}_{L^\infty}$ and $\norm{\cB^{\kappa,R}}_{L^2}$ are bounded uniformly in $\kappa$ and $R$. Similarly, \eqref{eq:uniform_coercivity} implies a uniform lower bound on $\cA^{\kappa,R}$. Summarizing, combining the Banach-Alaoglu theorem, compactness of weak solutions (see \Cref{lem:linear_compactness}), and a diagonalization argument with \eqref{eq:uniform_bound}, \eqref{eq:uniform_coercivity}, and \eqref{eq:uniform_dissipation} yields a limit function $f\colon \R^+\times \T^3\times \R^3 \to \R$ and sequences $\kappa_j \to 0^+$ and $R_j \to \infty$ such that the corresponding sequence of solutions $f_j$ converges to $f$ in the following senses:
    \begin{equation}\label{eq:convergences}
    \begin{aligned}
        &f_j \xrightharpoonup{\ast} f \quad\text{weak-star in }L^\infty(0,T;L^\infty_m)\text{ for each }T>0\text{;}\\
        &f_j \to f \quad \text{strongly in }L^2([0,T]\times \T^3\times \R^3) \text{ for each }T > 0\text{;}\\
        &f_j \to f \quad \text{pointwise almost everywhere in }[0,T]\times \T^3\times \R^3\text{; and }\\
        &\nabla_v f_j \rightharpoonup \nabla_v f \quad \text{weakly in }L^2_{\rm loc}((\tau,T)\times\T^3\times\R^3) \text{ for each }0<\tau<T.
    \end{aligned}
    \end{equation}
    From \eqref{eq:convergences} and $0 \le f_j \le 1$, we conclude that the limit function $f$ satisfies the Pauli exclusion bound $0\le f\le 1$. Moreover, $f$ has the qualitative regularity properties of a weak solution in the sense of \Cref{def:weak-solutions}. To conclude that $f$ is a weak solution, we argue that the coefficients in \eqref{eq:frozen_coefficients} converge. Indeed, the Lebesgue dominated convergence theorem and \eqref{eq:convergences} immediately yield
    \begin{equation*}
        \cA_j \to A[f(1-f)]  \quad \text{and} \quad \cB_j \to -f(1-f)\nabla a[f] \quad \text{pointwise almost everywhere in }\R^+\times\T^3\times \R^3.
    \end{equation*}
    Consequently, \eqref{eq:convergences} and the stability of weak solutions (see \Cref{lem:linear_stability}) imply $f$ is a weak solution to
    \begin{equation}
        (\partial_t + v \cdot\nabla_x)f = \nabla_v \cdot \left(A[f(1-f)]\nabla_v f - \nabla a[f]f(1-f)\right).
    \end{equation}
    We conclude that $f$ is a global weak solution to LFD in the sense of \Cref{def:weak-solutions}. Additionally, since each $f_j$ satisfies \Cref{thm:upper_bounds} and \Cref{t.nonvacuum}, the convergences in \eqref{eq:convergences} transfer the conclusions to $f$ as well. Finally, passing to the limit in \eqref{eq:uniform_dissipation}, the claimed dissipation bound holds:
    \begin{equation*}
        \int_0^T\iint A[f(1-f)]\nabla_v f\cdot\nabla_v f \dd v \dd x \dd t < \infty \qquad \text{for each }T>0.
    \end{equation*}
    Consequently, using the coefficient bounds in \Cref{lem:coefficient_bounds} once more,
    \begin{equation*}
        A[f(1-f)]\nabla_v f - \nabla a[f]f(1-f) \in L^2([0,T]\times \T^3\times \R^3) \quad \text{for each }T>0.
    \end{equation*}
    By duality, $(\partial_t + v\cdot\nabla_x )f \in L^2_{t,x}([0,T]\times \T^3; H^{-1}_v)$ as claimed. Lastly, we note that $f\in L^\infty([0,T];L^2_{x,v})$ and $f$ is continuous in the sense of distributions with $f(0) = f_{\rm in}$; this is a simple consequence of Aubin-Lions lemma with negative Sobolev spaces (see \Cref{lem:linear_stability}). Thus, 
    \begin{equation*}
        \lim_{t\to 0^+} f(t) = f_{\rm in} \qquad \text{weakly in $L^2_{x,v}$} \qquad \text{and} \qquad \norm{f_{\rm in}}_{L^2_{x,v}} \le \liminf_{t\to 0^+} \norm{f(t)}_{L^2_{x,v}}.
    \end{equation*}
    Therefore, passing to the limit as $j\to \infty$ and using the strong convergence of the initial data in \eqref{eq:uniform_dissipation}, 
    \begin{equation*}
        \sup_{0 < t < T} \norm{f(t)}_{L^2_{x,v}}^2 \le \norm{f_{\rm in}}_{L^2_{x,v}}^2 + 2\int_0^T \norm{f(t)}_{L^2_{x,v}}^2 \dd t.
    \end{equation*}
    Sending $T \to 0^+$, we find $\displaystyle\limsup_{t\to 0^+} \norm{f(t)}_{L^2_{x,v}} \le \norm{f_{\rm in}}_{L^2_{x,v}}$. We conclude that $\displaystyle\lim_{t\to 0^+} f(t) = f_{\rm in}$ strongly in $L^2$. This completes the proof.
    \end{proof}

    \section{Higher Regularity}\label{sec:regularity}

    In this section, we continue the proof of \Cref{thm:global_existence} by upgrading the regularity of the global weak solutions constructed in \Cref{prop:existence}. In particular, the existence of smooth solutions will be deduced from the following more precise result:
    \begin{proposition}\label{prop:higher_regularity}
        Suppose $f:[0,T]\times \T^3 \times \R^3 \to [0,1]$ is a weak solution to \eqref{e.LFD} with nonvacuum initial data in the sense of \Cref{d.nonvacuum} and an element of $L^\infty_m$ for some $m > 5$. Then, $f$ is a classical solution to \eqref{e.LFD}. Moreover, for any fixed $k \in \mathbb N$, there is $m_0$ such that, if $m \geq m_k$, then
        \begin{equation*}
            D^k_{t,x,v} f \in L^\infty.
        \end{equation*}
        The constant $m_k$ depends on $k$, $T$, $\|f_{\rm in}\|_{L^\infty_{m_0}}$, and the nonvacuum parameters $r$, $\delta$, and $v_0$.
    \end{proposition}

    \Cref{prop:higher_regularity} is proved by bootstrapping known regularity estimates for linear kinetic equations with minimal assumptions on the coefficients, as in \cite{henderson2017smoothing} for the classical Landau equation. 
    The proof is essentially identical to that in~\cite{henderson2017smoothing}; however, for completeness, we provide a proof.
    
\subsection*{Step 1: Collecting Known Local Regularity Results}

Local regularity estimates for linear kinetic equations are typically stated on a unit kinetic cylinder, using norms adapted to the kinetic scaling. For precise definitions of kinetic cylinders, distance, and associated H\"older semi-norms; see \Cref{ss.notation}.

In this step, we recall two interior estimates for equations of the form:
\begin{equation}\label{eq:linear_kinetic}
    (\partial_t + v\cdot \nabla_x )g
    = \nabla_v\cdot \left(\cA\nabla_v g\right) + \cB\cdot \nabla_v g + \cC g.
\end{equation}
We will assume a uniform coercivity condition on $A$ and that the coefficients are bounded, measurable:
\begin{equation}\label{eq:uniform_ellipticity}
        0 <\lambda I \le \cA(t,x,v) \le \Lambda I, \qquad \abs{\cB(t,x,v)} \le  \Lambda, \qquad \text{and} \qquad \abs{\cC(t,x,v)} \le \Lambda.
\end{equation}
The first main result we need is the kinetic De Giorgi-Nash-Moser estimate from \cite{golse2016}:
\begin{proposition}\label{prop:DeGiorgi}
    Suppose $g$ is a weak solution to \eqref{eq:linear_kinetic} in $Q_2$ and satisfies \eqref{eq:uniform_ellipticity} in $Q_2$ for some values of $\lambda$ and $\Lambda$. Then, there exists an $\alpha \in (0,1)$ depending only on $\lambda$ and $\Lambda$ so that $g\in C^{0,\alpha}_{\rm kin}(Q_1)$ and
    \begin{equation}
        [g]_{C^{0,\alpha}_{\rm kin}(Q_1)} \lesssim 
        \norm{g}_{L^\infty(Q_2)}.
    \end{equation}
    The implicit constant depends only on $\lambda$ and $\Lambda$.
\end{proposition}

The second main result we need is a kinetic Schauder estimate.  There are, by now, many to choose from, including~\cite{loher2022quantitative,henderson2017smoothing,imbert2018schauder}.  We synthesize them here, with further comments after the statement.
\begin{proposition}\label{prop:Schauder}

    Fix $\alpha \in (0,1)$.  Suppose $g$ is a solution to \eqref{eq:linear_kinetic} in $Q_2$ satisfying \eqref{eq:uniform_ellipticity} in $Q_2$ for some values of $\lambda$ and $\Lambda$.  Suppose that $\cA, \cB, \cC \in C^{0,\alpha}_{\rm kin}(Q_2)$.  Then, we have
        \begin{equation}
        [g]_{C^{2,\alpha}_{\rm kin}(Q_1)} \lesssim \left(1 + [\cA]_{C^{0,\alpha}_{\rm kin}(Q_2)}^\frac{2+\alpha}{\alpha}
        + [\cB]_{C^{0,\alpha}_{\rm kin}(Q_2)}^\frac{2+\alpha}{1+\alpha}
        + [\cC]_{C^{0,\alpha}_{\rm kin}(Q_2)}\right) \norm{g}_{L^\infty(Q_2)}.
    \end{equation}   
    Further suppose that for some $k \ge 3$ and some $\alpha \in (0,1)$,
    \begin{equation}\label{eq:Schauder_assumption}
        \norm{\cA}_{C^{k-3,\alpha}_{\rm kin}(Q_2)}
        + \norm{\nabla_v \cA}_{C^{k-3,\alpha}_{\rm kin}(Q_2)}
        + \norm{\cB}_{C^{k-3,\alpha}_{\rm kin}(Q_2)}
        + \norm{\cC}_{C^{k-3,\alpha}_{\rm kin}(Q_2)} \le \Lambda_k.
    \end{equation}
    Then there is $p_{k,\alpha}$, depending only on $k$ and $\alpha$ such that
    \begin{equation}\label{e.Schauder_k}
        [g]_{C^{k-1,\alpha}_{\rm kin}(Q_1)} \lesssim \left(1+\Lambda_k^{p_{k,\alpha}}\right) \norm{g}_{L^\infty(Q_2)}.
    \end{equation}
    The implicit constants depends only on $\lambda$, $\Lambda$, and $\alpha$ and, in the case of~\eqref{e.Schauder_k}, on $k$ as well.
\end{proposition}

We remark that while \Cref{prop:DeGiorgi} is explicitly proved for weak solutions by an energy method in \cite{golse2016}, the notion of solution for Schauder estimates is slightly more nuanced. Indeed, essentially everywhere Schauder estimates are shown for kinetic equations in non-divergence form, which do not admit forms of $H^1_{\rm kin}$-weak solutions. Note that the non-divergence form of Schauder estimates do not necessarily hold for weak solutions, unless one has an extra assumption of the $v$-derivative of the diffusion matrix. In divergence form, a standard mollification argument allows one to pass the estimate to weak solutions. 

Let us discuss the dependence of \Cref{prop:Schauder} on the H\"older norms of the coefficients.  The importance of this is that the H\"older norms of the coefficients will grow (polynomially) at large velocities; see~\Cref{l.Holder_coefficients}.  If we aim only for classical solutions, this plays no role in the analysis.  Indeed, we simply apply \Cref{prop:Schauder} locally to obtain $C^{2,\alpha}_{\rm kin, loc}$-regularity of $f$.  The fact that the dependence is polynomial only matters when we aim for {\em smoothness} of solutions via an iteration of \Cref{prop:Schauder}.  Indeed, at each application of \Cref{prop:Schauder}, we lose a finite amount of polynomial decay due to the polynomial growth of the norms of the coefficients.  If we suppose that $f\in L^\infty_{\rm rapid}$, however, we still obtain that $D^{k-1} f \in L^\infty_{\rm rapid}$ at each step.  This allows us, when verifying~\eqref{eq:Schauder_assumption}, to pass the derivative inside of the convolution defining the coefficients, e.g., $D^{k-3} A[f(1-f)] = A[D^{k-3} f(1-f)]$.

Second, the particular powers of the H\"older norms of the coefficients does not yet exist in the literature.  Most results are stated with implicit dependence.  To our knowledge, the only result with tracking explicit dependence is in~\cite{henderson2017smoothing}, where only the dependence on $[\cA]_{C^{0,\alpha}_{\rm kin}}$ is tracked (also, the power obtained there is non-optimal).  We note two things.  First, the particular power is unimportant to our arguments; we state it here only in case it is useful in the future.  Second, the powers in \Cref{prop:Schauder} may easily be obtained by scaling (see, e.g., the proof of~\cite[Theorem~4.5]{imbert2020smooth} in the nonlocal setting) and match precisely the powers in the elliptic and parabolic setting.

\subsection*{Step 2: A Change of Variables}

The local regularity estimates \Cref{prop:DeGiorgi} and \Cref{prop:Schauder} both assume the uniform ellipticity condition \eqref{eq:uniform_ellipticity} on a kinetic cylinder $Q_2$. However, the nonlocal nature of the coefficients in~\eqref{e.LFD} requires global estimates in $v$. Unfortunately, \eqref{eq:uniform_ellipticity} is not satisfied globally by the~\eqref{e.LFD} diffusion; accordingly, we introduce a global version of \eqref{eq:uniform_ellipticity} that will be satisfied by solutions to~\eqref{e.LFD}:
\begin{equation}\label{defn:coercivity_assumption}
    \lambda\left(\frac{\Pi(v)}{\brak{v}} + \frac{\Id - \Pi(v)}{\brak{v}^3}\right) \le A(t,x,v) \le \Lambda \left(\frac{\Pi(v)}{\brak{v}} + \frac{\Id - \Pi(v)}{\brak{v}^3}\right).
\end{equation}
Additionally, we will modify the boundedness assumption of the coefficients in \eqref{eq:uniform_ellipticity}:
\begin{equation}
\label{defn:bounded_coefficient_assumption}
    \abs{B(t,x,v)} \le \Lambda \brak{v}^{-2} \qquad \text{and} \qquad \abs{C(t,x,v)} \le \Lambda\brak{v}^{-5}.
\end{equation}
Our first goal is to prove a version of \Cref{prop:DeGiorgi} and \Cref{prop:Schauder} when \eqref{eq:uniform_ellipticity} is replaced by \eqref{defn:coercivity_assumption} and \eqref{defn:bounded_coefficient_assumption}. This requires an appropriate rescaling of the linear theory to the present situation. We change coordinates to preserve the structure of \eqref{eq:linear_kinetic}, but remove the anisotropy and degeneracy in the diffusion from \eqref{defn:coercivity_assumption}.

To this end, we follow the development in~\cite{henderson2017smoothing} (see also~\cite{CameronSilvestreSnelson} for earlier work involving this change of variables). We introduce a family of global coordinate changes on $\R_+ \times \R^6$. Fix a base point $z_0 = (t_0,x_0, v_0)$, around which we aim to apply the regularity estimates.  Define the parameters
\begin{equation}\label{eq:choice_r_and_S}
    r_{t_0} = \min\left(\frac{\sqrt{t_0}}{2},1\right) \quad \text{and} \quad S_{v_0} = \frac{1}{\brak{v_0}^{\sfrac12}} \left(\Id - \frac{v_0 \tens v_0}{\abs{v_0}^2}\right) + \frac{1}{\brak{v_0}^{\sfrac32}}\left(\frac{v_0 \tens v_0}{\abs{v_0}^2}\right).
\end{equation}
In the sequel, we often omit the dependence of $r$ and $S$ on $t_0$ and $v_0$. Then we let
\begin{equation}
    \Phi_{z_0}(z)
    := \left(t_0 + r^2 t, x_0 + r^3 Sx + r^2 t v_0, v_0 + r Sv\right).
\end{equation}
our strategy is to change coordinates using $\Phi_{z_0}$, apply a local regularity estimate in the new coordinate system, and invert the coordinate change.
The choice of $r$ ensures that $Q_1$, in the new coordinates, does not intersect with $t=0$; i.e., $t_0 + r^2 t > 0$.  
The choice of the linear transformation $S$ is exactly to combat the anisotropy and degeneracy of the diffusion in \eqref{defn:coercivity_assumption}.

Defining
\begin{equation}
    g_{z_0}(z) = g(\Phi_{z_0}(z)),
\end{equation}
we have that
\begin{equation}
    \left(\partial_t + v\cdot\nabla_x\right) g_{z_0}
    = \nabla_v \cdot \left(\overline{\cA}\nabla_v g_{z_0}\right) + \overline{\cB} \cdot \nabla_v g_{z_0} + \overline{\cC},
\end{equation}
where
\begin{equation}\label{eq:new_coefficients}
\begin{aligned}
    \overline{\cA}(z) &= S^{-1}A(\Phi_{z_0}(z))S^{-1},
    \qquad
    \overline{\cB}(z) &= rS^{-1}B(\Phi_{z_0}(z)),
    \qquad\text{ and }\qquad
    \overline{\cC}(z) &= r^2 C(\Phi_{z_0}(z)).
\end{aligned}
\end{equation}
Again, $\overline \cA$, $\overline \cB$, and $\overline \cC$ depend on $z_0$, but we often omit this notationally.

We now track how 
the assumptions \eqref{defn:coercivity_assumption} and \eqref{defn:bounded_coefficient_assumption} transform under the change of coordinates $\Phi$.

\begin{lemma}\label{l.transformed_coefficients}
    Consider any coefficients $\cA$, $\cB$, and $\cC$ on $[0,T] \times \Omega \times \R^3$ satisfying~\eqref{defn:coercivity_assumption}-\eqref{defn:bounded_coefficient_assumption}, and fix any $z_0 = (t_0,x_0,v_0) \in (0,T]\times B_1\times \R^3$. Then, $\overline \cA$, $\overline \cB$, and $\overline \cC$, defined by~\eqref{eq:new_coefficients}, satisfy~\eqref{eq:uniform_ellipticity} in $Q_2$ and, for any $m\geq 0$ and $z\in Q_2$,
    \begin{equation}
        |g_{z_0}(z)|
        \lesssim \vvo^{-m} \|g\|_{L^\infty_m}.
    \end{equation}.
\end{lemma}
    


\begin{proof}
The result about $\overline \cA$ can be proven exactly as \cite[Lemma~4.1]{CameronSilvestreSnelson}.  We omit the details.  Let us consider, next, $\overline \cB$.  We see, from~\eqref{eq:choice_r_and_S}, that $|S^{-1}| \lesssim \vvo^{\sfrac32}$.  Thus, recalling~\eqref{defn:coercivity_assumption},
\begin{equation}
    |\cB|
    \lesssim \vvo^{\sfrac32}|B(\Phi_{z_0}(z)|
    \lesssim \frac{\vvo^{\sfrac32}}{\brak{v_0 + r Sv}^{2}}
    \lesssim 1.
\end{equation}
The last inequality follows using the bounds on $r$, $S$, and the fact that $|v| \leq 2$.  The proof of the estimate of $\overline \cC$ is simpler, so we omit it.  This completes the proof.
\end{proof}

\subsection*{Step 3: Rescaled De Giorgi Estimates}

In the previous step, we shows that $g_{z_0}$ satisfies an equation to which the De Giorgi estimate \Cref{prop:DeGiorgi} applies.  In other words, we find that
\begin{equation}
    [g_{z_0}]_{C^{0,\alpha}_{\rm kin}(Q_1)}
    \lesssim \|g_{z_0}\|_{L^\infty(Q_2)},
\end{equation}
where the implicit constant above depends only on those in~\eqref{defn:coercivity_assumption}-\eqref{defn:bounded_coefficient_assumption}.  We need the following result, which converts regularity under $\Phi_{z_0}$ to the original variables.
\begin{lemma}\label{l.transformed Holder}
    Fix $\alpha\in (0,1)$, $\mu_0 > 0$, $\mu \in [\mu_0,2]$, $k\geq 0$, and $z_0$.  Given any function $g: Q_{\mu r} \to \R$, we have that
    \begin{equation}
        \frac{r^{k+\alpha}}{\vvo^\frac{3(k+\alpha)}{2}}[g]_{C^{k,\alpha}_{\rm kin}(Q_{\mu\tilde r}(z_0))}
            \leq [g_{z_0}]_{C^{k,\alpha}_{\rm kin}(Q_{\mu})}
            \leq r^{k+\alpha} [g]_{C^{k,\alpha}_{\rm kin}(Q_{\mu r}(z_0))}
    \end{equation}
    where we have defined $\tilde r = r\vvo^{-\sfrac32}.$
\end{lemma}
\begin{proof}
    This is essentially due to~\cite[equation~(2.6)]{HST2020landau}; however, the details are not fully fleshed out there, so we include a proof. For simplicity, let us take the case $k=0$. The result follows if we show that
    \begin{equation}
        \frac{1}{r}d\left(\Phi_{z_0}(z_1),\Phi_{z_0}(z_2)\right)
        \leq d(z_1,z_2)
        \leq \frac{\vvo^{\sfrac32}}{r}d\left(\Phi_{z_0}(z_1),\Phi_{z_0}(z_2)\right).
    \end{equation}
    Recall the definition~\eqref{defn:kinetic_distance} of $d$.  This is a straightforward computation:
    \begin{equation}
    \begin{split}
        &\frac{1}{r}d\left(\Phi_{z_0}(z_1),\Phi_{z_0}(z_2)\right)
        \\&\qquad
        = |t_2 - t_1|^{\sfrac12}
            + |S(x_2-x_1)  -  (t_2-t_1)S v_1|^{\sfrac13}
            + |Sv_2 - Sv_1|
        \\&\qquad
        \leq
            |t_2 - t_1|^{\sfrac12}
            + \vvo^{-\sfrac16}|(x_2-x_1)  -  (t_2-t_1) v_1|^{\sfrac13}
            + \vvo^{-\sfrac12}|v_2 - v_1|
        \leq
            d(z_1,z_2).
    \end{split}
    \end{equation}
    A similar argument establishes the other inequality.  This completes the proof.
\end{proof}

Putting together the above with \Cref{prop:DeGiorgi}, we obtain the following result.

\begin{lemma}\label{l.rescaled_DeGiorgi}
    Fix $z_0$.  Suppose $g$ is a weak solution to~\eqref{eq:linear_kinetic} in $Q_{2r}(z_0)$ with coefficients satisfying~\eqref{defn:coercivity_assumption}-\eqref{defn:bounded_coefficient_assumption}. Suppose that $g \in L^\infty_m$ for some $m \in \R$.  Then there exists $\alpha\in (0,1)$, depending only on $\lambda, \Lambda$, such that
    \begin{equation}
        [g]_{C^{0,\alpha}_{\rm kin}(Q_\frac{r}{2}(z_0))}
        \lesssim \frac{1}{r^\alpha \vvo^{m - \frac{3\alpha}{2}}} \norm{g}_{L^\infty_m(Q_{2r}(z_0))}.
    \end{equation}
    Recall that $\tilde r = r \vvo^{-\sfrac32}$ and $r = \min\{1, \sfrac{\sqrt{t_0}}2\}$.
\end{lemma}
\begin{proof}
We first establish a ``very local'' estimate.  Fix any $z_1 \in Q_{\frac{r}{2}}(z_0)$.  In view of \Cref{l.transformed_coefficients}, we may apply \Cref{prop:DeGiorgi} to $g_{z_1}$ on $Q_2$ to find
    \begin{equation}
        [g_{z_1}]_{C^{0,\alpha}_{\rm kin}(Q_{\sfrac{1}{2}})}
        \lesssim \|g_{z_1}\|_{L^\infty(Q_1)}
        \lesssim \langle v_1\rangle^{-m} \|g\|_{L^\infty_m(Q_1(z_1))}
        \lesssim \vvo^{-m} \|g\|_{L^\infty_m(Q_{2r}(z_0))}.
    \end{equation}
    In the last inequality, we used that $\vvo \approx \langle v_1\rangle$ and that $Q_r(z_1) \subset Q_{2r}(z_0)$ because $z_1 \in Q_{\frac{r}{2}}(z_0)$. 
    Then, applying \Cref{l.transformed Holder}, we find
    \begin{equation}\label{e.bad local DG}
        [g]_{C^{0,\alpha}_{\rm kin}(Q_{\frac{\tilde r}{2}}(z_1))}
        \lesssim \frac{\vvo^{-m +\frac{3\alpha}{2}}}{r^\alpha}\|g\|_{L^\infty_m(Q_3(z_0))}.
    \end{equation}
    
Let us now fix two points $z_1,z_2 \in Q_r(z_0)$.  If $d(z_1,z_2) \leq \tilde r$, then~\eqref{e.bad local DG} yields
\begin{equation}
    \frac{|g(z_1) - g(z_2)|}{d(z_1,z_2)^\alpha}
    \lesssim \frac{\|g\|_{L^\infty_m}}{r^\alpha \vvo^{m-\frac{3\alpha}{2}}}.
\end{equation}
On the other hand, if $d(z_1,z_2) \geq \tilde r$, then we find
\begin{equation}
    \frac{|g(z_1) - g(z_2)|}{d(z_1,z_2)^\alpha}
    \lesssim \frac{\vvo^{-m} \|g\|_{L^\infty_m}}{\tilde r^\alpha}
    = \frac{\|g\|_{L^\infty_m}}{r^\alpha \vvo^{m-\frac{3\alpha}{2}}}.
\end{equation}
The proof is complete.    
\end{proof}

\subsection*{Step 3: Rescaled Schauder Estimates}

When the coefficient $\cA$, $\cB$, and $\cC$ are themselves H\"older continuous, we may apply the Schauder estimates of \Cref{prop:Schauder}, properly rescaled.  We write a general estimate here, and show how to apply it to the nonlinear problem in the next step.  The key point in the nonlinear estimate is that the H\"older regularity of the solution $f$ to~\eqref{e.LFD} established in Step 2 yields H\"older regular of the coefficients of~\eqref{e.LFD}.

\begin{lemma}\label{l.rescaled_Schauder}
Fix $z_0$.  Suppose $g$ solves~\eqref{eq:linear_kinetic} in $Q_2(z_0)$ with coefficients satisfying~\eqref{defn:coercivity_assumption}-\eqref{defn:bounded_coefficient_assumption}. Suppose that $g \in L^\infty_m$ for some $m \in \R$.  Assume, further, that there exists $\alpha\in (0,1)$ with $\cA, \cB, \cC \in C^{0,\alpha}_{\rm kin}(Q_2(z_0))$.  Then
    \begin{equation}
        [g]_{C^{2,\alpha}_{\rm kin}(Q_{\frac{r}{2}}(z_0))}
        \lesssim \frac{1
        + r^{2+\alpha} [\cA]_{C^{0,\alpha}_{\rm kin}(Q_{2r}(z_0))}^\frac{2+\alpha}{\alpha}
        + r^\frac{\alpha(2+\alpha)}{1+\alpha} [\cB]_{C^{0,\alpha}_{\rm kin}(Q_{2r}(z_0))}^\frac{2+\alpha}{1+\alpha}
        + r^\alpha [\cC]_{C^{0,\alpha}_{\rm kin}(Q_{2r}(z_0))}
        }{r^{2+\alpha}\vvo^{m - 3 - \frac{3\alpha}{2}}}\norm{g}_{L^\infty_m}.
    \end{equation}
The implicit constant above depends on $\alpha$, $\lambda$, $\Lambda$, $[\cA]_{C^{0,\alpha}_{\rm kin}}$, $[\cB]_{C^{0,\alpha}_{\rm kin}}$, and $[\cC]_{C^{0,\alpha}_{\rm kin}}$.  Recall that $\tilde r = r \vvo^{-\sfrac32}$ and $r = \min\{1,\sfrac{\sqrt{t_0}}2\}$.
\end{lemma}
\begin{proof}
As we did in \Cref{l.rescaled_DeGiorgi}, we first establish a ``very local'' version of the estimate.  Fix any $z_1 \in Q_{\frac{r}{2}}(z_0)$.  We begin by observing how the H\"older norms of the coefficients depends on $v_0$.  Applying \Cref{l.transformed Holder} yields
\begin{equation}
    [\cA_{z_0}]_{C^{0,\alpha}_{\rm kin}(Q_1)}
    \lesssim r^\alpha [\cA]_{C^{0,\alpha}_{\rm kin}(Q_{r}(z_1))}
    \lesssim r^\alpha [\cA]_{C^{0,\alpha}_{\rm kin}(Q_{2r}(z_0))}.
\end{equation}
Similarly for the other coefficients.  Then, applying \Cref{prop:Schauder} to $g_{z_0}$, we find
\begin{equation}
    [g_{z_0}]_{C^{2,\alpha}_{\rm kin}(Q_{\sfrac12}(z_0))}
    \lesssim
    \left(1
        + r^{2+\alpha} [\cA]_{Q_{2r}(z_0)}^\frac{2+\alpha}{\alpha}
        + r^\frac{\alpha(2+\alpha)}{1+\alpha} [\cB]_{Q_{2r}(z_0)}^\frac{2+\alpha}{1+\alpha}
        + r^\alpha [\cC]_{Q_{2r}(z_0)}
        \right)
        \|g_{z_1}\|_{L^\infty(Q_1)}.
\end{equation}
Applying \Cref{l.transformed Holder} and \Cref{l.transformed_coefficients} yields
\begin{equation}\label{e.bad local Schauder}
    [g]_{C^{2,\alpha}_{\rm kin}(Q_{\frac{\tilde r}{2}}(z_0))}
    \lesssim
        \frac{1
        + r^{2+\alpha} [\cA]_{Q_{2r}(z_0)}^\frac{2+\alpha}{\alpha}
        + r^\frac{\alpha(2+\alpha)}{1+\alpha} [\cB]_{Q_{2r}(z_0)}^\frac{2+\alpha}{1+\alpha}
        + r^\alpha [\cC]_{Q_{2r}(z_0)}
        }{r^{2+\alpha} \vvo^{m-3-\frac{3\alpha}{2}}}
        \|g\|_{L_m^\infty(Q_{2r}(z_0))}.
\end{equation}

Let us now extend this to an estimate on $Q_{\frac{r}{2}}(z_0)$.  Fix any two points $z_1, z_2 \in Q_\frac{r}{2}(z_0)$.  We focus on $D^2_vg$; however, a similar argument applies for $(\partial_t + v\cdot\nabla_x)g$.  If $d(z_1,z_2) < \sfrac{\tilde r}{2}$, we may apply~\eqref{e.bad local Schauder} directly:
\begin{equation}
\begin{split}
    \frac{|D^2_v g(z_1) - D^2_v g(z_2)|}{d(z_1,z_2)^\alpha}
    &\leq [g]_{C^{2,\alpha}_{\rm kin}(Q_{\frac{\tilde r}{2}}(z_1))}
    \\&
    \lesssim
        \frac{1
        + r^{2+\alpha} [\cA]_{Q_{2r}(z_0)}^\frac{2+\alpha}{\alpha}
        + r^\frac{\alpha(2+\alpha)}{1+\alpha} [\cB]_{Q_{2r}(z_0)}^\frac{2+\alpha}{1+\alpha}
        + r^\alpha [\cC]_{Q_{2r}(z_0)}
        }{r^{2+\alpha} \vvo^{m-3-\frac{3\alpha}{2}}}
        \|g\|_{L_m^\infty(Q_{2r}(z_0))}.
\end{split}
\end{equation}
This is exactly the desired estimate. 
We, thus, only consider the case $d(z_1,z_2) \geq \sfrac{\tilde r}{2}$. In this case, we argue as we did in \Cref{l.rescaled_DeGiorgi}.  There is, however, an additional wrinkle: we do not have a direct estimate on the $C^2_{\rm kin}$-norm of $g$ and must instead interpolate.  For this, we recall the following interpolation estimate: for any $r,s>0$, and $z\in \R_+ \times \R^3 \times \R^3$,
\begin{equation}
    |D^2_v g(z)|
    \lesssim
        \frac{1}{s^2} \|g\|_{L^\infty_m(Q_s(z_0))}
        + s^\alpha [g]_{C^{2,\alpha}_{\rm kin}(Q_s(z))}.
\end{equation}
This can be proven using standard methods, and is stated in, e.g., \cite[Lemma 2.2]{henderson2017smoothing}, \cite[Proposition 2.10]{imbert2018schauder}, or related works.  Applying it at $z_1$ and $z_2$ with radius $\sfrac{\tilde r}{2}$, we find
\begin{equation}
\begin{split}
    \frac{|D^2_v g(z_1) - D^2_v g(z_2)|}{d(z_1,z_2)^\alpha}
    &\lesssim
        \frac{1}{\tilde r^\alpha} \left( \frac{1}{\tilde r^2 \vvo^m} \|g\|_{L^\infty_m(Q_{2r}(z_0))}
        + \tilde r^\alpha [f]_{C^{2,\alpha}_{\rm kin}(Q_{\frac{\tilde r}{2}}(z_0))} \right)
    \\&
    \lesssim
        \frac{1}{\tilde r^{2+\alpha} \vvo^m} \|g\|_{L^\infty_m(Q_{2r}(z_0))}
        + [f]_{C^{2,\alpha}_{\rm kin}(Q_{\frac{\tilde r}{2}}(z_0))}. 
\end{split}
\end{equation}
Combining this with~\eqref{e.bad local Schauder} and recalling the definition of $\tilde r$ completes the proof.
\end{proof}

\subsection*{Step 4: The Nonlinear Bootstrap}

We now apply the estimates developed above to our equation. 

\begin{proof}[Proof of \Cref{prop:higher_regularity}]
Fix $\tau_0 \in (0,T]$.  We show regularity of $f$ on $[\tau_0, T]\times \T^3 \times \R^3$.  Throughout the proof, it is understood that all estimates depend on $\tau_0$ and $T$ unless otherwise indicated.  Additionally, let us (re-)define the nonvacuum parameters of \Cref{d.nonvacuum} as $\delta_{\rm nv}$, $r_{\rm nv}$, and $v_{\rm nv}$ to avoid some notational clash.

Let us first note that, by \Cref{thm:upper_bounds}, 
\begin{equation}\label{e.c080401}
    \sup_{t\in[0,T]} \|f(t)\|_{L^\infty_{\bar m}}
    \lesssim \|f_{\rm in}\|_{L^\infty_{\bar m}},
\end{equation}
where $\bar m = \min\{m,5\}$.  (Note: we take this minimum to avoid constants blowing up when $m\gg 1$).  The estimate~\eqref{e.c080401} does not depend on $\tau_0$.

From~\eqref{e.c080401} and \Cref{lem:coefficient_bounds}, we directly obtain the $\Lambda$ in~\eqref{defn:bounded_coefficient_assumption} and the upper bound in~\eqref{defn:coercivity_assumption}.  We observe that $\Lambda$ depends only on the same quantities as the implicit constant in~\eqref{e.c080401}. The lower bound in~\eqref{defn:coercivity_assumption} follows from \Cref{t.nonvacuum}.  Thus, we obtain $\lambda$ that depends on $\bar m$, $\|f_{\rm in}\|_{L^\infty_{\bar m}}$, $\delta_{\rm nv}$, $r_{\rm nv}$, and $v_{\rm nv}$.

As a result, we may apply \Cref{l.rescaled_DeGiorgi} to deduce that there exists $\alpha>0$ such that, for any $z_0$,
\begin{equation}
    [f]_{C^{0,\alpha}_{\rm kin}(Q_r(z_0))}
    \lesssim \frac{1}{r^\alpha \vv^{m - \frac{3\alpha}{2}}} \|f_{\rm in}\|_{L^\infty_m}.
\end{equation}
The constant $\alpha$ depends only on $\bar m$, $\|f_{\rm in}\|_{L^\infty_{\bar m}}$, $\delta_{\rm nv}$, $r_{\rm nv}$, and $v_{\rm nv}$.  From this point on, all estimates are understood to depend on these quantities even when not explicitly stated.  Considering all possible cylinders, we find
\begin{equation}\label{e.c080601}
    \|\vv^{m- \frac{3\alpha}{2}} f\|_{C^{0,\alpha}_{\rm kin}([\sfrac{\tau}{2},T]\times \T^3 \times \R^3)}
    \lesssim \|f_{\rm in}\|_{L^\infty_m}.
\end{equation}
Up to decreasing $\alpha$, we may assume that
\begin{equation}\label{e.c080602}
    m > 5 + \frac{3\alpha}{2}.
\end{equation}
This is not necessary but does simplify some details below.  Let us also notice that, from here on, the dependence on $\|f_{\rm in}\|_{L^\infty_{\rm in}}$ becomes nonlinear.  As a result, we absorb it into the implicit constant, and all further implicit constants are understood to depend on it even when not explicitly stated.

We now show that $f$ is a classical solution.  Using \Cref{l.Holder_coefficients}, we see that the coefficients of~\eqref{e.LFD} are H\"older continuous of order $\alpha':=\sfrac{2\alpha}{3}$.  As a result, \Cref{l.rescaled_Schauder} yields
\begin{equation}
\begin{split}
    &\vvo^{m - 3 - \frac{3\alpha'}{2}}[f]_{C^{2,\alpha'}_{\rm kin}(Q_{r}(z_0))}
    \\&
    \lesssim \frac{1}{r^{2+\alpha'}}
        + \Big[A[f(1-f)]\Big]_{C^{0,\alpha'}_{\rm kin}(Q_{2r}(z_0))}^\frac{2+\alpha'}{\alpha'}
        + \frac{\big[b[f]\big]_{C^{0,\alpha'}_{\rm kin}(Q_{2r}(z_0))}^\frac{2+\alpha'}{1+\alpha'}}{r^\frac{2+\alpha'}{1+\alpha'}}
        + \frac{[f(1-f)]_{C^{0,\alpha'}_{\rm kin}(Q_{2r}(z_0))}
        }{r^2}
    \\&
    \lesssim
        \frac{1}{r^{2+\alpha'}}
        + \frac{1}{\vvo^{\left(\frac{3}{\alpha} + 1\right)\left(1 - \frac{\alpha}{3}\right)}}
        + \frac{1}{r^\frac{2+\alpha'}{1+\alpha'}\vvo^{\frac{6+2\alpha}{3+2\alpha}\left(2 - \frac{2\alpha}{9}\right)}}
        + \frac{1}{r^2\vvo^{m - \alpha}}
    \lesssim
        \frac{1}{r^{2+\alpha'}}
        + \frac{1}{\vvo^{\frac{3}{\alpha} - \frac{\alpha}{3}}}.
\end{split}
\end{equation}
Explicitly, the first inequality is due to \Cref{l.rescaled_Schauder}, the second uses \Cref{l.Holder_coefficients} and~\eqref{e.c080601}, and the last follows from an application of Young's inequality as well as the lower bound~\eqref{e.c080602} on $m$. We also used the simple identity that $(\sfrac3\alpha + 1)(1 - \sfrac\alpha3) = \sfrac3\alpha - \sfrac\alpha3$.  Varying over all cylinder, we find
\begin{equation}\label{e.c080701}
    \left[\vv^{m - 3 - \alpha} f\right]_{C^{2,\alpha'}_{\rm kin}([\tau/2,T]\times \T^3 \times \R^3)}
    \lesssim \frac{1}{r^{2 + \alpha'}}.
\end{equation}

The argument for higher regularity follows using similar arguments.  Indeed, we can pass the regularity of $f$ in~\eqref{e.c080701} to the coefficients (cf.~\Cref{l.Holder_coefficients}), apply the transformation $\Phi_{z_0}$ to make $A[f(1-f)]$ uniformly elliptic (\Cref{l.transformed_coefficients}), then re-apply the Schauder estimates \Cref{prop:Schauder} with $k=4$. We can then interate this process to $k=5,6,\dots$.  At each step, we lose some decay in $v$ of $f$ (e.g., we begin with $f\in L^\infty_m$, but deduce~\eqref{e.c080701}).  Hence, we may iterate this argument indefinitely if $f\in \Lrapid$.  If not, there is $k_m$, depending on $m$ as well as all previous parameters, such that we may iterate this argument to deduce $f \in C^{k_m}_{\rm kin}$.  This completes the proof.
\end{proof}
    
\appendix

\section{Weak Solutions to Linear Kolmogorov-Type Equations}\label{appendix:linear}

In this section, we recall some basic theory of linear Kolmogorov-type equations with bounded measurable diffusion and divergence-form sources, i.e. equations of the form
\begin{equation}\label{eq:linear_kolmogorov}
        (\partial_t + v\cdot \nabla_x) f = \nabla_v \cdot (\mathcal A \nabla_v f) + \nabla_v \cdot \cB\qquad \text{for }(t,x,v) \in [0,T]\times \T^3\times \R^3, 
\end{equation}
where the coefficients satisfy $\cA\in L^\infty([0,T]\times \T^3 \times \R^3)$ and $\cB\in L^2([0,T]\times\T^3\times \R^3)$.  

We begin with a precise statement of the notion of weak solution used herein:
\begin{definition}[Weak Solution]\label{defn:linear_weak_solution}
    We say that $f$ is a weak solution to \eqref{eq:linear_kolmogorov} with initial data $f_{\rm in}\in L^2_{x,v}$ provided $f \in L^\infty([0,T];L^2_{\rm loc}(\T^3\times\R^3))$, $\nabla_v f \in L^2_{\rm loc}([0,T]\times\T^3\times\R^3)$, and 
    \begin{equation}
        \iiint f\;(\partial_t + v\cdot\nabla_x)\psi \dd v\dd x\dd t  + \iint f_{\rm in}\psi(0) \dd v \dd x = \iiint \nabla_v \psi \cdot \left[\cA\nabla_v f + \cB\right] \dd v \dd x \dd t 
    \end{equation}
    for each $\psi \in C^\infty_c([0,T)\times \T^3\times \R^3)$. 
\end{definition}

The results below---especially pertaining to existence and uniqueness---are part of the folklore of kinetic theory, making it difficult to find precise references. Nevertheless, the versions we quote are not optimal and do not justify more than a brief discussion.

\subsection{Existence, Uniqueness, and Compactness}

Our first linear result contains existence and uniqueness of weak solutions and the standard {\em a priori} estimates that come with them. 
\begin{lemma}\label{lem:linear_existence}
    Suppose that $\cA$ is a measurable matrix field satisfying
    \begin{equation*}
        \lambda\, \Id \le \cA(t,x,v)\le \Lambda\, \Id \qquad \text{for each }(t,x,v)\in [0,T]\times \T^3 \times \R^3.
    \end{equation*}
    Suppose further that
    \begin{equation*}
        \norm{f_{\rm in}}_{L^2_{x,v}} \le \Lambda \qquad \text{and} \qquad \norm{\cB}_{L^2([0,T]\times \T^3\times\R^3)} \le \Lambda.
    \end{equation*}
    Then, \eqref{eq:linear_kolmogorov} admits a unique weak solution with initial data $f_{\rm in}$. The constructed solution further satisfies $f \in C([0,T];L^2_{x,v})$ and $\nabla_v f\in L^2_{t,x,v}$ with estimates depending only on $\lambda$ and $\Lambda$.
\end{lemma}

In some sense, \Cref{lem:linear_existence} is easiest to prove from scratch. The existence portion follows from the standard Galerkin method used for parabolic equations. Since the equation is linear, strong compactness (which is especially cumbersome in the $t$ and $x$ variables) is unnecessary. Additionally, the $L^2$-norm is monotone decreasing when $\cB = 0$, which implies uniqueness, provided $f$ is justified as an admissible test function in the weak formulation. See also \cite{AuscherImbertNiebel} for existence and uniqueness for a closely related family of equations. The last point is continuity in time, which is a consequence of the following observation:

\begin{lemma}\label{lem:linear_continuity}
    Suppose that $f\in L^2(0,T;L^2_xH^1_v)$ and $(\partial_t + v\cdot\nabla_x) f \in L^2(0,T;L^2_xH^{-1}_v)$. Then, $f$ is strongly continuous after modification on a measure zero set of times, i.e. $f \in C([0,T];L^2_{x,v})$.
\end{lemma}

Continuity in time comes from a type of kinetic Lions-Magenes formula: formally, one sees that
\begin{equation}
    \frac{\dd}{\dd t} \norm{f}_{L^2_{x,v}}^2 = 2\iint f(\partial_t + v\cdot \nabla_x f) \dd v \dd x \le 2\norm{f}_{L^2_xH^1_v}\norm{(\partial_t+v\cdot \nabla_x)f}_{L^2_xH^{-1}_v}.
\end{equation}
\Cref{lem:linear_continuity} then follows from a straightforward modification of the proof in the parabolic setting: Truncating in $v$ and mollifying in $t$ and $x$, the same identity holds. Estimating directly, the $L^\infty_tL^2_{x,v}$ norm of the difference between two mollification scales is bounded uniformly in $t$ by the norms on the right hand side. Since a uniformly convergent sequence of continuous functions converges to a continuous limit, $f$ is continuous after modification on a measure zero set of times.

Next, we isolate useful compactness and stability properties of weak solutions to \eqref{eq:linear_kolmogorov}. These estimates will be used frequently even when the diffusion coefficients are not uniformly elliptic, so we allow for degenerate diffusion:
\begin{lemma}\label{lem:linear_compactness}
    Suppose that $\cA_n$ and $\cB_n$ are families of measurable coefficients such that: $\cA_n$ are uniformly bounded in $L^\infty([0,T]\times \T^3\times\R^3)$; $\cB_n$ are uniformly bounded in $L^2([0,T]\times \T^3\times \R^3)$; and
    \begin{equation*}
        \cA_n(t,x,v) \ge \lambda(t,x,v)\,\Id , \quad \text{where } \quad \lambda(t,x,v)\text{ is uniformly positive on  compact subsets of }(0,T)\times \T^3\times\R^3.
    \end{equation*}
    Further suppose that $f_{\rm in}^n$ is a family of initial data that is uniformly bounded in $L^2_{\rm loc}(\T^3\times\R^3)$. Then, any corresponding family $f_n\colon [0,T]\times \T^3\times \R^3$ of weak solutions to \eqref{eq:linear_kolmogorov} is locally compact in $L^2_{t,x,v}$. 
    
    If additionally $\brak{v}^\sigma f_{\rm in}^n$ is uniformly bounded in $L^2_{x,v}$ for some $\sigma > 0$, then $\set{f_n}$ is compact in $L^2(0,T;L^2_{x,v})$.
\end{lemma}

Compactness in $v$ follows from $\nabla_v f \in L^2_{\rm loc}([0,T]\times \T^3\times \R^3)$ by the usual $L^2$ energy estimate. Compactness in $x$ follows from $L^2$-based averaging lemmas for \eqref{eq:linear_kolmogorov}. For example, see \cite{bouchut2002hypoelliptic}, where $\nabla_x^{1/3} f \in L^2_{\rm loc}([0,T]\times\T^3\times \R^3)$ is shown. Finally, compactness in $t$ follows from the Aubin-Lions lemma. This provides local compactness. For global compactness, one uses that the equation propagates $L^2$ moments to obtain tightness.

\begin{lemma}\label{lem:linear_stability}
    Assuming the uniform bounds in \Cref{lem:linear_compactness}, further suppose that $\cA_n \to \cA$ and $\cB_n \to \cB$ pointwise almost everywhere in $(0,T)\times \T^3\times \R^3$ and that $f_{\rm in}^n \to f_{\rm in}$ strongly in $L^2(\T^3 \times \R^3)$. Then, any corresponding limit point $f$ satisfies $f \in C_w([0,T];L^2_{x,v})$ and $f$ is a weak solution to \eqref{eq:linear_kolmogorov} with coefficients $\cA$ and $\cB$ and initial data $f_{\rm in}$.
\end{lemma}

The proof is entirely straightforward. Passing to the limit in $B_n\nabla f_n$ is justified since $B_n$ converges strongly in $L^1_{loc}$ and $\nabla f_n$ converges weakly. One ensures the correct initial data is obtained because the solutions are evidently strongly compact in $C(0,T;H^{-2}_{\rm loc})$, i.e. sense of distributions.

\bibliographystyle{abbrv}
\bibliography{ref}

\end{document}